\documentclass[a4paper,12pt,twoside,reqno,tbtags]{amsart}
\pdfoutput=1
\usepackage[english]{babel}

\newif\ifwias
\wiastrue
\wiasfalse

\ifwias
	\usepackage{wiaspreprint}
\fi

\usepackage{amsmath,amssymb}
\usepackage{mathtools}
\usepackage{color}
\usepackage{upgreek}
\usepackage{subfigure}
\usepackage{caption}
\usepackage{booktabs}
\usepackage{cancel}
\usepackage{multirow}
\usepackage{enumitem}

\usepackage[colorinlistoftodos,textsize=footnotesize]{todonotes}

\usepackage[
        activate={true,nocompatibility},
        final,
        kerning=true,
        spacing=true
        ]{microtype}
\microtypecontext{spacing=nonfrench}
\usepackage{fixltx2e}

\usepackage{tikz}
\usepackage{pgfplots}
\usepackage{pgfplotstable}
\pgfplotsset{compat=1.7}

\usepackage{mathtools}
\mathtoolsset{showonlyrefs}

\usepackage[T1]{fontenc}
\usepackage{lmodern}

\newif\iffinal
\finaltrue
\finalfalse

\iffinal
  \usepackage[useplots]{tikzconditional}
  \usepackage[ruled,vlined]{algorithm2e}
\else
  \usepackage[noexternal]{tikzconditional}
  \usepackage[ruled,vlined,norelsize]{algorithm2e}
\fi

\tikzsetexternalprefix{figures/}
\tikzsetfigurename{ttals-}

\numberwithin{equation}{section}

\theoremstyle{definition}
\newtheorem{definition}{Definition}[section]

\theoremstyle{remark}
\newtheorem{remark}[definition]{Remark}

\theoremstyle{plain}
\newtheorem{theorem}[definition]{Theorem}
\newtheorem{corollary}[definition]{Corollary}
\newtheorem{lemma}[definition]{Lemma}
\newtheorem{proposition}[definition]{Proposition}

\renewcommand*{\phi}{\varphi}
\let\ntheta\theta
\renewcommand*{\theta}{\vartheta}
\renewcommand*{\rho}{\varrho}

\renewcommand*{\Gamma}{\varGamma}
\renewcommand*{\Delta}{\varDelta}
\renewcommand*{\Theta}{\varTheta}
\renewcommand*{\Lambda}{\varLambda}
\renewcommand*{\Xi}{\varXi}
\renewcommand*{\Pi}{\varPi}
\renewcommand*{\Sigma}{\varSigma}
\renewcommand*{\Upsilon}{\varUpsilon}
\renewcommand*{\Phi}{\varPhi}
\renewcommand*{\Psi}{\varPsi}
\renewcommand*{\Omega}{\varOmega}

\newcommand*{\latinabbr}[1]{#1}
\newcommand*{\ie}{\latinabbr{i.e.}} 

\newcommand*{\defeq}{\ensuremath{\coloneqq}}

\newcommand*{\dd}{\ensuremath{\mathrm{d}}}
\newcommand*{\dx}[1]{\ensuremath{\,\dd{#1}}}

\DeclareMathOperator{\id}{id}

\DeclareMathOperator{\supp}{supp}
\DeclareMathOperator{\diam}{diam}
\DeclareMathOperator{\ddiv}{div}

\newcommand*{\essinf}{\mathop{\mathrm{ess\,inf}}\displaylimits}
\newcommand*{\esssup}{\mathop{\mathrm{ess\,sup}}\displaylimits}
\newcommand*{\jump}[1]{\ensuremath{ [\![ #1 ]\!] }}

\newcommand*{\abs}[2][{}]{\ensuremath{#1\lvert#2#1\rvert}}
\newcommand*{\norm}[2][{}]{\ensuremath{#1\lVert#2#1\rVert}}

\newcommand*{\mset}[3][{}]{\ensuremath{#1\{ #2 \,;\: #3#1\}}}

\def\beq{\begin{equation}}
\def\eeq{\end{equation}}

\def\mcA{{\mathcal{A}}}
\def\mcB{{\mathcal{B}}}

\def\mcN{{\mathcal{N}}}
\def\mcF{{\mathcal{F}}}
\def\mcI{{\mathcal{I}}}
\def\mcT{{\mathcal{T}}}

\def\mcR{{\mathcal{R}}}
\def\mcS{{\mathcal{S}}}

\def\mcV{{\mathcal{V}}}

\def\mcX{{\mathcal{X}}}

\def\bbR{\mathbb{R}}
\def\bbN{\mathbb{N}}

\def\ainf{{\check a}}
\def\asup{{\hat a}}

\def\ainf{{\check a}}
\def\asup{{\hat a}}

\def\bfbeta{{\boldsymbol{\beta}}}
\def\bfsigma{{\boldsymbol{\sigma}}}

\def\Hthrh{H^{\tau_{\theta\rho}}}

\def\Nmc{{N_{\mathrm{MC}}}}

\newcommand{\dgam}[1]{\mathrm{d}\gamma(#1)}
\newcommand{\dgamthrh}[1]{\mathrm{d}\gamma_{\vartheta\rho}(#1)}

\SetKw{DownTo}{downto}
\SetKw{UpTo}{upto}
\newcommand{\disc}{_{\mathrm{disc}}}
\newcommand{\Det}{_{\mathrm{det}}}
\newcommand{\param}{_{\mathrm{param}}}
\DeclareMathOperator{\est}{est}
\DeclareMathOperator{\err}{err}
\DeclareMathOperator{\Solve}{Solve}
\DeclareMathOperator{\Estimate}{Estimate}
\DeclareMathOperator{\Mark}{Mark}
\DeclareMathOperator{\Refine}{Refine}

\DeclareMathOperator{\res}{res}

\ifwias
\title[Adaptive lognormal TT SGFEM]
{Adaptive stochastic Galerkin FEM for lognormal coefficients in hierarchical tensor representations}
\date{June 29, 2018 (revision: November 1, 2018)}

\pagefootright{Berlin, June 29, 2018/rev. November 1, 2018}
\nopreprint{2515}
\nopreyear{2018}

\author[M. Eigel, M. Marschall, M. Pfeffer, R. Schneider]{
Martin Eigel\footnote{Weierstrass Institute\\
		   			  Mohrenstr. 39\\
					  10117 Berlin\\
					  Germany\\ 
					  E-Mail: martin.eigel@wias-berlin.de \\ 
					  \phantom{E-Mail: }manuel.marschall@wias-berlin.de}, 
Manuel Marschall\addmark{\,\,}{1}, 
Max Pfeffer\footnote{Max Planck Institute for Mathematics in the Sciences\\ 
					 Inselstr. 22\\
					 04103 Leipzig\\ 
					 Germany\\ 
					 E-Mail: max.pfeffer@mis.mpg.de}, 
Reinhold Schneider\footnote{TU Berlin\\
							Stra\ss e des 17. Juni 136\\
							10623 Berlin\\
							Germany\\ 
							E-Mail: schneidr@math.tu-berlin.de} 
}

\subjclass[2010]{%
 35R60, 
 47B80, 
 60H35, 
 65C20, 
 65N12, 
 65N22, 
 65J10
}
\keywords{Partial differential equations with random coefficients, tensor representation, tensor train, uncertainty quantification, stochastic finite element methods, log-normal, adaptive methods, ALS, low-rank, reduced basis methods}

\else
\title[Adaptive lognormal TT SGFEM]
{Adaptive stochastic Galerkin FEM for lognormal coefficients in hierarchical tensor representations}
\date{\today}

\author[M.\ Eigel]{Martin Eigel}
\address{Weierstrass Institute\\Mohrenstrasse 39\\D-10117 Berlin\\Germany}
\email{martin.eigel@wias-berlin.de}

\author[M.\ Marschall]{Manuel Marschall}
\address{Weierstrass Institute\\Mohrenstrasse 39\\D-10117 Berlin\\Germany}
\email{manuel.marschall@wias-berlin.de}

\author[M.\ Pfeffer]{Max Pfeffer}
\address{Max Planck Institute for Mathematics in the Sciences
		\\Inselstr. 22\\D-04103 Leipzig\\Germany}
\email{max.pfeffer@mis.mpg.de}

\author[R.\ Schneider]{Reinhold Schneider}
\address{TU Berlin\\Stra{\ss}e des 17. Juni 136\\D-10623 Berlin\\Germany}
\email{schneidr@math.tu-berlin.de}

\subjclass[2010]{%
 35R60, 
 47B80, 
 60H35, 
 65C20, 
 65N12, 
 65N22, 
 65J10
}
\keywords{Partial differential equations with random coefficients, tensor representation, tensor train, uncertainty quantification, stochastic finite element methods, lognormal, adaptive methods, ALS, low-rank, reduced basis methods}

\fi

\begin{document}



\selectlanguage{english}
\begin{abstract}
Stochastic Galerkin methods for non-affine coefficient representations are known to cause major difficulties from theoretical and numerical points of view.
In this work, an adaptive Galerkin FE method for linear parametric PDEs with lognormal coefficients discretized in Hermite chaos polynomials is derived.
It employs problem-adapted function spaces to ensure solvability of the variational formulation.
The inherently high computational complexity of the parametric operator is made tractable by using hierarchical tensor representations.
For this, a new tensor train format of the lognormal coefficient is derived and verified numerically.
The central novelty is the derivation of a reliable residual-based a posteriori error estimator.
This can be regarded as a unique feature of stochastic Galerkin methods.
It allows for an adaptive algorithm to steer the refinements of the physical mesh and the anisotropic Wiener chaos polynomial degrees.
For the evaluation of the error estimator to become feasible, a numerically efficient tensor format discretization is developed.
Benchmark examples with unbounded lognormal coefficient fields illustrate the performance of the proposed Galerkin discretization and the fully adaptive algorithm.
\end{abstract}

\maketitle

\section{Introduction}
\label{sec:Introduction}

In the thriving field of Uncertainty Quantification (UQ), efficient numerical methods for the approximate solution of random PDEs have been a topic of vivid research.
As common benchmark problem, one often considers the Darcy equation (as a model for flow through a porous medium) with different types of random coefficients in order to assess the efficiency of a numerical approach.
Two important properties are the length of the expansion of random fields, which often directly translates to the number of independent random variables describing the variability in the model, and the type of dependence on these random variables.
The affine case with uniform random variables has been studied extensively, since it represents a rather simple model which can easily be treated with standard methods.
Opposite to that, the lognormal case with Gaussian random variables is quite challenging, from the analytical as well as numerical point of view.
A theory for the solvability of linear elliptic PDEs with respective unbounded coefficients (and hence a lack of uniform ellipticity) in a variational formulation was only developed recently in~\cite{Gittelson2011a,MR2576514,Mugler2013thesis}.
Computationally, the problems quickly become difficult or even intractable with many stochastic dimensions, which might be required to accurately represent the stochasticity in the random field expansion.
This paper is concerned with the development of an efficient numerical method for this type of problems.

While popular sample-based Monte Carlo methods obtain dimension-independent convergence rates, these are rather low despite often encountered higher regularity of the parameter dependence.
Moreover, such methods can only be used to evaluate functionals of the solution (QoIs = quantities of interest) and an a posteriori error control usually is not feasible reliably.
Some recent developments in this field can e.g. be found in~\cite{herrmann2017multilevel,herrmann2017qmc,kazashi2017quasi,kuo2017multilevel} for the model problem with a lognormal coefficient.
Some ideas on a posteriori adaptivity for Monte Carlo methods can e.g. be found in~\cite{eigel2016adaptive,detommaso2018continuous}.

An alternative are functional (often called spectral) approximations, which for instance are obtained by Stochastic Collocation (SC)~\cite{babuvska2007stochastic,nobile2008sparse,nobile2008anisotropic}, the related Multilevel Quadraure (MLQ)~\cite{harbrecht2017quasi,harbrecht2016multilevel} and Stochastic Galerkin (SG\footnote{we usually use SGFEM for Stochastic Galerkin FEM}) methods.
The latter in particular is popular in the engeneering sciences since it can be perceived as an extension of classical finite element methods (FEM).
These approaches provide a complete parameter to solution map based on which e.g. statistical moments of the stochastic solution can be evaluated.
Notably, the regularity of the solution can be exploited in order to obtain quasi optimal convergence rates.
However, the number of random variables and nonlinear parameter representations have a significant impact on the computational feasibility and techniques for a model order reduction are required.
Collocation methods with pointwise evaluations in the parameter space are usually constructed either based on some a priori knowledge or by means of an iterative refinement algorithm which takes into account the hierarchical surplus on possible new discretization levels.
While these approaches work reasonably well, methods for a reliable error control do not seem immediate since the approximation relies only on interpolation properties.
Nevertheless, for the affine case and under certain assumptions, first ideas were recently presented in~\cite{guignard2017posteriori}.

The computationally more elaborate Stochastic Galerkin methods carry out an orthogonal projection with respect to the energy norm onto a discrete space which is usually spanned by a tensor basis consisting of FE basis functions in physical space and polynomial chaos polynomials orthogonal with respect to the joint parameter measure in (stochastic) parameter space.
The use of global polynomials is justified by the high (analytic) regularity of the solution map with respect to the parameters~\cite{MR2728424,MR3163401,MR2763359}.
However, in particular the large computational cost of Galerkin methods make adaptivity and model reduction techniques a necessity.

In order to achieve this, different paths have been pursued successfully.
As a first approach, sparse approximations as in~\cite{EGSZ,EGSZ2,EM} or~\cite{MR3177362,crowder2018efficient,bespalov2016efficient} with either a residual based or an hierarchical a posteriori error estimators can be computed.
Here, the aim is to restrict an exponentially large discrete basis to the most relevant functions explictly by iteratively constructing a quasi-optimal subset.
In~\cite{EGSZ2}, convergence of the adaptive algorithm could be shown.
Moreover, adjoint based error estimators are considered in~\cite{bryant2015error,prudhomme2015adaptive}.

As a second approach, an adaptive discretization in hierarchical tensor representations can be derived as described in~\cite{EPS}.
These modern compression formats have lately been investigated intensively in the numerical analysis community~\cite{Hackbusch2014,Bachmayr2016,nouy2016low,Oseledets2009}.
It has been examined that with appropriate assumptions the curse of dimensionality can be alleviated, particularly so when employed with typical random PDE problems in UQ, see~\cite{dolgov2015polynomial,dolgov2017hybrid} for examples with sample-based reconstruction strategies.
Such representations can be understood as an implicit model order reduction technique, closely related to (but more general than e.g.) reduced basis methods.

In the mentioned adaptive approaches, the FE mesh for the physical space and the parametric polynomial chaos space are adapted automatically with respect to the considered problem.
In the case of tensor approximations, also the ranks of the representation are adjusted.

However, in all adaptive SGFEM research, so far only the affine case with uniform elliptic coefficient has been considered.
In this paper, we extend the ASGFEM approach developed in~\cite{EPS} to the significantly more involved case of lognormal (non-affine) coefficients.
This poses several severe complications analytically and numerically.
Analytical aspects have recently been tackled in~\cite{MR2576514,MR2649152,Mugler2013thesis,MR3163401}.
Numerically, in particular computationally efficient Galerkin methods are quite diffucult to construct for this case and have not been devised.
Compression techniques and adaptivity most certainly are required in order to make these problems tractable with SGFEM, as described in this paper.
Of particular interest is the construction of a computable a posteriori error estimator, which also greatly benefits from using tensor formats.
In order to obtain a well-posed discretization, problem adapted spaces according to the presentation in~\cite{MR2805155} are used.

Main contributions of this work are a representation of the coefficient in the tensor train (TT) format, the operator discretization in tensor format and the derivation of an reliable residual based error estimatator.
This then serves as the basis for an adaptive algorithm which steers all discretization parameters of the SGFEM.
The performance of the proposed method is demonstrated with some benchmark problems.
Here, the used field models are not artificially bounded or shifted away from zero.

We point out that, to our knowledge, an SGFEM for the lognormal case so far has only been practically computed in the weighted function space setting in~\cite{mugler2013convergence} for a small 1d model as proof of concept.
Moreover, there has not been any adaptive numerical method with reliable a posteriori error estimation as derived in this work.
However, we note that our approach relies on the assumption that the coefficient is discretized sufficiently accurately and hence the related discretization error can be neglected.
In practice, this can be ensured with high probability by sampling the error of the discrete coefficient.
Additionally, since constants in the error bound can become quite large, we interpret the error estimate as a refinement indicator.

It should be noted that a functional adaptive evluation of the forward map allows for the derivation of an explicit adaptive Bayesian inversion with functional tensor representations as in~\cite{EMS}.
The results of the present work lay the ground for a similar approach with a Gaussian prior assumption.
This will be the topic of future research.
Moreover, the described approach enables to construct SGFEM with arbitrary densities (approximated in hierarchical tensor formats).
This generalization should also be examined in more detail in further research.
Lastly, while sparse discretizations seem infeasible for the lognormal coefficient, a transformation~\cite{ullmann2012efficient} yields a convection-diffusion formulation of the problem with affine parameter dependence, which then again is amenable to an adaptive sparse SGFEM.
This direction is examined in~\cite{EG}.

The structure of this paper is as follows: We first introduce the setting of parametric PDEs with our model problem in Section~\ref{sec:parametric-pde}.
The variational formulation in problem dependent weighted function spaces and the finite element (FE) setting are described in Section~\ref{sec:Galerkin Discretization}.
The employed tensor formats and the tensor representations of the coefficient and the differential operator are examined in Section~\ref{sec:Tensor Decomposition Formats}.
As a central part, in Section~\ref{sec:Error Estimates} we derive the a posteriori error estimators and define a fully adaptive algorithm in Section~\ref{sec:Fully Adaptive Algorithm}, including efficient ways to compute error indicators for physical and stochastic refinement.
Numerical examples in Section~\ref{sec:Numerical Experiments} illustrate the performance of the presented method and conclude the paper.

\section{Setting and Discretization}
\label{sec:parametric-pde}

In the following, we introduce the considered model problem formally, present its weak formulation and describe the employed discretizations in finite dimensional function spaces.
We closely follow the presentations in~\cite{MR2805155,MR3163401,MR2649152} regarding the lognormal problem in problem-dependent function spaces.
In~\cite{EPS}, a related formulation for the solution and evaluation of a posteriori error estimators for parametric PDEs with affine coefficient fields in hierarchical tensor formats is derived.

\subsection{Model problem}
\label{sec:model-problem}

We assume some bounded domain $D \subset\bbR^d$, $d=1,2,3$, with Lipschitz boundary $\partial D$ and a probability space $(\Omega, \mathcal F, \mathbb{P})$.
For $\mathbb{P}$-almost all $\omega\in\Omega$, we consider the random elliptic problem
\begin{equation}
\label{eq:model}
\left\{ \begin{aligned}
-\ddiv(a(x,\omega)\nabla u(x,\omega)) &= f(x)\quad\text{ in $D$},\\
u(x,\omega) &= 0\quad\qquad\text{on $\partial D$.}
\end{aligned} \right.
\end{equation}

The coefficient $a:D\times\Omega\mapsto\bbR$ denotes a lognormal, isotropic diffusion coefficient, i.e., $\log(a)$ is an isotropic Gaussian random field.

\begin{remark}
The source term $f\in L^2(D)$ is assumed deterministic.
However, it would not introduce fundamental additional difficulties to also model $f$ and the boundary conditions as stochastic fields not correlated to the coefficient $a(x, \omega)$ as long as appropriate integrability of the data is given.
\end{remark}

For the coefficient $a(x,\omega)$ of~\eqref{eq:model}, we assume a Karhunen-Lo\`eve type expansion of $b:=\log(a)$ of the form
\begin{equation}
\label{eq:a-KL}
b(x,\omega) = \sum_{\ell=1}^\infty b_\ell(x)Y_\ell(\omega),\qquad x\in D,\quad  \mathbb{P}\text{- almost all } \omega\in\Omega.
\end{equation}
Here, the parameter vector $Y = (Y_\ell)_{\ell\in\bbN}$ consists of independent standard normal Gaussian random variables in $\bbR$.
Then, by passing to the image space $(\bbR^\bbN,\mcB(\bbR^\bbN),\gamma)$ with the Borel $\sigma$-algebra $\mcB(\bbR^\bbN)$ of all open sets of $\bbR^\bbN$ and the Gaussian product measure
\begin{align}
\gamma := \bigotimes_{\ell\in\bbN} \gamma_\ell\quad &\text{with}\quad \gamma_\ell:=\gamma_1:=\mcN_1:=\mcN(0,1) \quad \\
&\text{and}\quad \dx{\gamma_1}(y_\ell)=\frac{1}{\sqrt{2\pi}}\exp(-y_\ell^2/2)\dx{y_\ell},
\end{align} 
we can consider the parameter vector $y=(y_\ell)_{\ell\in\bbN}=(Y_\ell(\omega))_{\ell\in\bbN}$, $\omega\in\Omega$.

For any sequence $\bfbeta\in\ell^1(\bbN)$ with
\begin{equation}
  \beta_\ell := \norm{b_\ell}_{L^\infty(D)} \qquad\text{and}\qquad \bfbeta = (\beta_\ell)_{\ell\in\bbN},
\end{equation}
we define the set
\begin{equation}
\Gamma_\bfbeta := \bigg\{ y\in\bbR^\bbN:\ \sum_{\ell=1}^\infty \beta_\ell\abs{y_\ell} < \infty \bigg\}.
\end{equation}
The set $\Gamma_\bfbeta$ of admissible parameter vectors is $\gamma$-measurable and of full measure.
\begin{lemma}[\cite{MR3163401} Lem. 2.1]
\label{lem:gamma-prop}
For any sequence $\bfbeta\in\ell^1(\bbN)$, there holds
\[
\Gamma_\bfbeta\in\mcB(\bbR^\bbN)\quad\text{and}\quad \gamma(\Gamma_\bfbeta) = 1.
\]
\end{lemma}
For any $y\in\Gamma_\bfbeta$, we define the deterministic parametric coefficient
\begin{equation}
\label{eq:a-param}
a(x,y) = \exp(b(x,y)) = \exp\left(\sum_{\ell=1}^\infty b_\ell(x) y_\ell \right),\qquad x\in D.
\end{equation}
This series converges in $L^\infty(D)$ for all $y\in\Gamma_\bfbeta$.
\begin{lemma}[\cite{MR3163401} Lemma 2.2]
\label{lem:a-bounds}
For all $y\in\Gamma_\bfbeta$, the diffusion coefficient~\eqref{eq:a-param} is well-defined and satisfies
\begin{equation}
0 < \ainf(y) := \essinf_{x\in D} a(x,y)\leq a(x,y) \leq \esssup_{x\in D} a(x,y) =: \asup(y) < \infty,
\end{equation}
with
\begin{align}
\asup(y) \leq \exp\bigg(\sum_{\ell=1}^\infty \beta_\ell\abs{y_\ell} \bigg)\quad\text{and}\quad
\ainf(y) \geq \exp\bigg(-\sum_{\ell=1}^\infty \beta_\ell\abs{y_\ell} \bigg).
\end{align}
\end{lemma}

Due to Lemmas~\ref{lem:gamma-prop} and~\ref{lem:a-bounds}, we consider $\Gamma=\Gamma_\bfbeta$ as the parameter space instead of $\bbR^\bbN$.
By Lemma~\ref{lem:a-bounds}, the stochastic coefficient $a(x,y)$ is well defined, bounded from above and admits a positive lower bound for almost all $y\in\Gamma$.
Thus, the equations~\eqref{eq:model} and~\eqref{eq:a-param} have a unique solution $u(y)\in \mcX$ for almost all $y\in\Gamma$.

Let $\mcX:=H^1_0(D)$ denote the closed subspace of functions in the Sobolev space $H^1(D)$ with vanishing boundary values in the sense of trace and define the norm
\begin{equation}
\norm{v}_\mcX := \left(\int_D \abs{\nabla v(x)}^2\dx{x} \right)^{1/2}.
\end{equation}
We denote by $\langle\cdot,\cdot\rangle=\langle\cdot,\cdot\rangle_{\mcX^\ast,\mcX}$ the duality pairing of $\mcX^\ast$ and $\mcX$ and consider $f$ as an element of the dual $\mcX^\ast$.

For any $y\in\Gamma$, the variational formulation of~\eqref{eq:model} reads: find $u(y)\in \mcX$ such that
\begin{align}
\label{eq:weak-param-eq}
B(u(y),v;y) &= \langle f, v \rangle,\qquad \text{for all } v\in \mcX,
\intertext{where $B:\mcX \times \mcX \times \Gamma \to \bbR$ is defined by}
B(u(y),v;y) &:= \int_D a(x,y)\nabla u(y) \cdot\nabla v\dx{x}.
\end{align}
Hence, pathwise existence and uniqueness of the solution $u(y)$ is obtained by the Lax-Milgram lemma due to uniform ellipticity for any fixed $y\in\Gamma$.
In particular, for all $y\in\Gamma$,~\eqref{eq:weak-param-eq}, it holds
\begin{equation}
  \label{eq:u-regularity}
\norm{u(y)} \leq \frac{1}{\ainf(y)}\norm{f}_{\mcX^\ast},
\end{equation}
with some $0 < \ainf(y)\leq a(x,y)$ on $D$.
The integration of~\eqref{eq:weak-param-eq} over $\Gamma$ with respect to the standard normal Gaussian measure $\gamma$ does not lead to a well-defined problem since the coefficient $a(x,y)$ is not uniformly bounded in $y\in\Gamma$ and not bounded away from zero.
Hence, a more involved approach has to be pursued, which is elaborated in Section~\ref{sec:weak-formulation}.
Alternative results for this equation were presented in~\cite{MR2805155, MR2649152,MR2576514}.

The formulation of~\eqref{eq:model} as a parametric deterministic elliptic problem with solution $u(y)\in \mcX$ for each parameter $y\in\Gamma$ reads
\begin{equation}
-\ddiv(a(x,y)\nabla u(x,y)) = f(x)\quad\text{for } x\in D,\quad u(x,y) = 0\quad\text{for } x\in\partial D.
\end{equation}

\section{Variational Formulation and Discretization}
\label{sec:Galerkin Discretization}

This section is concerned with the introduction of appropriate function spaces required for the discretization of the model problem.
In particular, a problem-adapted probability measure is introduced which allows for a well-defined formulation of the weak problem rescaled polynomial chaos basis.

\subsection{Problem-adapted function spaces}
\label{sec:problem-adapted}

Let $\mcF$ be the set of finitely supported multi-indices
\begin{equation}
\mcF := \mset{\mu\in\bbN_0^\infty}{\abs{\supp\mu} <\infty}\quad\text{where}\quad \supp\mu := \mset{m\in\bbN}{\mu_m\neq 0}.
\end{equation}
A full tensor index set of order $M\in\bbN$ is defined by
\begin{align}
\Lambda &:= \{(\mu_1,\ldots,\mu_M,0,\ldots) \in \mathcal F : \mu_m = 0,\ldots, d_m-1, \ m = 1,\ldots,M \} \\
&\simeq \Lambda_1 \times \ldots \times \Lambda_M \times \{0\} \ldots \, \subset\mcF, \\
\intertext{with complete index sets of size $d_m$ given by}
\Lambda_m &:= \{0,\ldots,d_m-1\}, \quad m = 1,\ldots,M.
\end{align}
For any such subset $\Lambda\subset\mcF$, we define $\supp\Lambda := \bigcup_{\mu\in\Lambda} \supp\mu \subset \bbN$.

We denote by $(H_n)_{n=0}^\infty$ the orthonormal basis of $L^2(\bbR, \gamma_m) = L^2(\bbR, \gamma_1)$ with respect to the standard Gaussian measure consisting of Hermite polynomials $H_n$ of degree $n\in\bbN_0$ on $\bbR$.
An orthogonal basis of $L^2(\Gamma, \gamma)$ is obtained by tensorization of the univariate polynomials, see~\cite{MR2805155,Ullmann2008thesis}.
To reduce notation, we drop the explicit dependency on the sigma-algebra which is always assumed to be rich enough.
For any multi-index $\mu\in\mcF$, the tensor product polynomial $H_\mu \defeq \bigotimes_{m=1}^\infty H_{\mu_m}$ in $y\in\Gamma$ is expressed as the finite product
\begin{equation}
H_\mu(y) = \prod_{m=1}^ \infty H_{\mu_m}(y_m) = \prod_{m\in\supp \mu}H_{\mu_m}(y_m).
\end{equation}
%
For practical computations, an analytic expression for the triple product of Hermite polynomials can be used.
\begin{lemma}[\cite{Ullmann2008thesis,malliavin2015stochastic}]
\label{lem:tripleH}
For $\mu,\nu,\xi\in\mcF$, $m\in\bbN$, it holds
\begin{equation}
\label{eq:Hermform}
H_{\nu_m} H_{\mu_m} = \sum_{\xi_m = 0}^{\mathrm{min}(\nu_m,\mu_m)} \kappa_{\nu_m,\mu_m,\nu_m + \mu_m - 2\xi_m} H_{\nu_m + \mu_m - 2\xi_m},
\end{equation}
where for $\eta_m = \nu_m + \mu_m - 2\xi_m$
\begin{align}
\label{eq:betacoeffs}
\kappa_{\nu_m,\mu_m,\eta_m} &:= \int_\bbR H_{\nu_m}(y_m) H_{\mu_m}(y_m) H_{\eta_m}(y_m) \dx{\gamma_m}(y_m) \\
&=
\begin{cases}
\frac{\sqrt{\nu_m!\mu_m!\eta_m!}}{\xi_m!(\nu_m - \xi_m)!(\mu_m - \xi_m)!}, \quad \begin{split}
&\nu_m + \mu_m - \eta_m \text{ is even and } \\
&|\nu_m - \mu_m| \leq \eta_m \leq \nu_m + \mu_m, \\
\end{split}\\
0, &\text{otherwise}.
\end{cases}
\end{align}
\end{lemma}

\begin{lemma}[\cite{MR2855645, Mugler2013thesis}]
\label{lem:expH}
Let $Y\sim\mcN_1$, $t\in\bbR$ and $X=\exp(t Y) \in L^2(\bbR,\gamma)$.
The expansion of $X=\exp(t Y)$ in Hermite polynomials is given by
\begin{equation}
X=\exp(t Y) = \sum_{n\in\bbN_0} c_n H_n \quad\text{with}\quad c_n = \frac{t^n}{\sqrt{n!}}\exp(t^2/2).
\end{equation}
\end{lemma}

We recall some results from~\cite{MR2805155} required in our setting.
Let $\bfsigma = (\sigma_m)_{m\in\bbN}\in\exp(\ell^1(\bbN))$ and define
\begin{equation}
\gamma_\bfsigma := \bigotimes_{m=1}^\infty \gamma_{\sigma_m} := \bigotimes_{m=1}^\infty \mcN_{\sigma_m^2} := \bigotimes_{m=1}^\infty \mcN(0, \sigma_m^2).
\end{equation}
Then, $\dx{\mcN_{\sigma_m^2}} = \zeta_{\bfsigma,m}\dx{\mcN_1}$ where
\begin{equation}
 \zeta_{\bfsigma,m}(y_m) := \frac{1}{\sigma_m}\exp\left(-\frac12(\sigma_m^{-2}-1)y_m^2\right)
\end{equation}
is the one-dimensional Radon-Nikodym derivative of $\gamma_{\sigma_m}$ with respect to $\gamma_m$, i.e., the respective probability density.
We assume that the sequence $\bfsigma$ depends exponentially on $\bfbeta = (\beta_m)_{m\in\bbN}$ and some $\rho\in\bbR$, namely
\begin{equation}
\sigma_m(\rho) := \exp(\rho\beta_m),\quad m\in\bbN,
\end{equation}
and define
\begin{align}
\gamma_\rho &:= \gamma_{\bfsigma(\rho)} \qquad\text{ and }\qquad \zeta_{\rho,m} := \zeta_{\bfsigma(\rho),m}.
\end{align}
%
%
By multiplication, this yields the multivariate identity
\begin{equation}
  \dx{\gamma_\rho}(y) = \zeta_{\rho}(y) \dx{\gamma}(y)
  \quad\text{with}\quad
  \zeta_{\rho}(y) = \prod_{m=1}^\infty \zeta_{\rho, m}(y_m).
\end{equation}
%
%
A basis of orthonormal polynomials with respect to the weighted measure $\gamma_\rho$ can be defined by the transformation
\begin{equation}
\tau_\rho:\bbR^\infty\to\bbR^\infty,\quad (y_m)_{m\in\bbN}\mapsto \left(e^{-\rho\beta_m}y_m\right)_{m\in\bbN}.
\end{equation}
Then, for all $v\in L^2(\Gamma,\gamma)$,
\begin{equation}
\int_\Gamma v(y)\gamma(\dx{y}) = \int_\Gamma v(\tau_\rho(y))\dx{\gamma_\rho}(y).
\end{equation}
We define the scaled Hermite polynomials $H^{\tau_\rho}_\mu := H_\mu\circ\tau_\rho$.

%

\begin{remark}
Lemmas \ref{lem:tripleH} and \ref{lem:expH} are also valid with the transformed multivariate Hermite polynomials $H^{\tau_\rho}$.
In particular, $\kappa_{\xi_m, \nu_m, \mu_m}$ does not change under transformation and the expansion in Lemma~\ref{lem:expH} holds by substituting $t\in\bbR$ with $\sigma_m t$ in the corresponding dimension $m\in\bbN$.
\end{remark}

\subsection{Weak formulation in problem-dependent spaces}
\label{sec:weak-formulation}

In order to obtain a well-posed variational formulation of~\eqref{eq:model} on $L^2(\Gamma,\gamma;\mcX)$, we follow the approach in~\cite{MR2805155} and introduce a measure $\gamma_\rho$ which is stronger than $\gamma$ and assume integrability of $f$ with respect to this measure.
For $\rho > 0$ and $0\leq \theta < 1$, let the bilinear form $B_{\theta\rho}:\mcV_{\theta\rho}\times\mcV_{\theta\rho}\to\bbR$ be given by
\begin{equation}
\label{eq:Btr}
B_{\theta\rho}(w,v) := \int_\Gamma\int_D a(x,y)\nabla w(x,y)\cdot\nabla v(x,y)\dx{x}\gamma_{\theta\rho}(\dx{y}).
\end{equation}
The solution space is then defined as the Hilbert space
\begin{equation}
\mcV_{\theta\rho} := \mset{v:\Gamma\to\mcX\quad \mcB(\Gamma)\text{-measurable}}{B_{\theta\rho}(v,v)<\infty},
\end{equation}
endowed with the inner product $B_{\theta\rho}(\cdot,\cdot)$, the induced energy norm $\norm{v}_{\theta\rho}:=B_{\theta\rho}(v,v)^{1/2}$ for $v\in\mcV_{\theta\rho}$ and the respective dual pairing $\langle\cdot,\cdot\rangle_{\theta\rho}$ between $\mcV_{\theta\rho}^\ast$ and $\mcV_{\theta\rho}$.
The different employed spaces are related as follows.
\begin{lemma}[\cite{MR2805155} Prop. 2.43]
For $0<\theta<1$,
\begin{equation}
  \label{eq:embedding}
L^2(\Gamma,\gamma_\rho;\mcX) \subset \mcV_{\theta\rho} \subset L^2(\Gamma,\gamma;\mcX)
\end{equation}
are continuous embeddings.
\end{lemma}
It can be shown that the bilinear form $B_{\theta\rho}(\cdot,\cdot)$ is $\mcV_{\theta\rho}$-elliptic in the sense of the following Lemma.
\begin{lemma}[\cite{MR2805155} Lem. 2.41, 2.42]
\label{lem:elliptic}
For $w,v\in L^2(\Gamma,\gamma_\rho;\mcX)$,
\begin{equation}
\label{eq:Vtr-boundedness}
\abs{B_{\theta\rho}(w,v)} \leq \hat c_{\vartheta \rho} \norm{w}_{L^2(\Gamma,\gamma_\rho;\mcX)} \norm{v}_{L^2(\Gamma,\gamma_\rho;\mcX)}
\end{equation}
and for $v \in L^2(\Gamma,\gamma;\mcX)$,
\begin{equation}
\label{eq:Vtr-coercivity}
B_{\theta\rho}(v,v) \geq \check c_{\vartheta \rho} \norm{v}^2_{L^2(\Gamma,\gamma;\mcX)}.
\end{equation}
\end{lemma}
Moreover, we assume that $f$ is such that the linear form
\begin{equation}
F_{\theta\rho}(v) := \int_\Gamma\int_D f(x)v(x,y)\gamma_{\theta\rho}(\dx{y})
\end{equation}
is well-defined.
For $F_{\theta\rho}\in\mcV_{\theta\rho}^\ast$,~\eqref{eq:Vtr-boundedness} and~\eqref{eq:Vtr-coercivity} in particular lead to the unique solvability of the variational problem in $\mcV_{\theta\rho}$,
\begin{equation}
\label{eq:weak-Vtrs}
B_{\theta\rho}(u,v) = F_{\theta\rho}(v)\qquad\text{ for all } v\in\mcV_{\theta\rho},
\end{equation}
and $u\in\mcV_{\theta\rho}$ is the unique solution of~\eqref{eq:model}.

\subsection{Deterministic discretization}
\label{sec:deterministic-disc}

We discretise the deterministic space $\mcX$ by a conforming finite element space
$\mcX_p(\mcT) := \mathrm{span} \{\varphi_i\}_{i=1}^N \subset \mcX$ of degree $p$ on some simplicial regular mesh $\mcT$ of domain D with the set of faces $\mcS$ (i.e., edges for $d=2$) and basis functions $\varphi_i$.

In order to circumvent complications due to an inexact approximation of boundary values, we assume that $D$ is a polygon.
We denote by $P_p(\mcT)$ the space of piecewise polynomials of degree $p$ on the triangulation $\mcT$.
The assumed FE discretization with Lagrange elements of order $p$ then satisfies $\mcX_p(\mcT)\subset P_p(\mcT)\cap C(\overline{\mcT})$.
For any element $T\in\mcT$ and face $F\in\mcS$, we set the entity sizes $h_T := \diam T$ and $h_F := \diam F$.
Let $n_F$ denote the exterior unit normal on any face $F$.
The jump of some $\chi\in H^1(D;\bbR^d)$ on $F = \overline{T_1}\cap\overline{T_2}$ in normal direction $\jump{\chi}_F$ is then defined by
\begin{equation}
\jump{\chi}_F := \chi\vert_{T_1}\cdot n_F - \chi\vert_{T_2}\cdot n_F.
\end{equation}
By $\omega_T$ and $\omega_F$ we denote the element and facet patches defined by the union of all elements which share at least a vertex with $T$ or $F$, respectively.
Consequently, the Cl\'ement interpolation operator $I\colon \mcX \to \mcX_p(\mcT)$ satisfies
\begin{align}
\label{eq:clement}
  \begin{array}{rlr}
    \norm{v- I v}_{L^2(T)} &\leq c_{\mathcal T} h_T \abs{v}_{\mcX,\omega_T} &\text{for } T\in\mcT,\\
    \norm{v- I v}_{L^2(F)} &\leq c_{\mathcal S} h_F^{1/2} \abs{v}_{\mcX,\omega_F}& \text{for } F\in\mcS,
  \end{array}\\
\end{align}
where the seminorms $\abs{\;\cdot\;}_{\mcX,\omega_T}$ and $\abs{\;\cdot\;}_{\mcX,\omega_F}$ are the restrictions of $\norm{\;\cdot\;}_{\mcX}$ to $\omega_T$ and $\omega_F$, respectively.

The fully discrete approximation space with $\abs{\Lambda}<\infty$ is given by
\begin{equation}
\mcV_N := \mcV_N(\Lambda;\mcT,p) := \mset[\bigg]{v_N(x,y) = \sum_{\mu\in\Lambda} v_{N,\mu}(x)\Hthrh_\mu(y)}{ \ v_{N,\mu}\in \mcX_p(\mcT)}, 
\end{equation}
and it holds $\mcV_N(\Lambda;\mcT,p)\subset\mcV_{\theta\rho}$.
The Galerkin projection of $u$ is the unique $u_N\in\mcV_N(\Lambda;\mcT,p)$ which satisfies
\begin{equation}
\label{eq:Galerkin-uN}
B_{\theta\rho}(u_N,v_N) = F_{\theta\rho}(v_N)\qquad\text{for all } v_N\in\mcV_N(\Lambda;\mcT,p).
\end{equation}
We define a tensor product interpolation operator $\mcI\colon L^2(\Gamma,\gamma;\mcX) \to \mcV_N(\Lambda;\mcT,p)$
for $v = \sum_{\mu \in \mathcal F} v_\mu H_\mu \in L^2(\Gamma,\gamma;\mcX)$ by setting
\begin{equation}
\mcI v := \sum_{\mu \in \Lambda} (I v_\mu) H_\mu.
\end{equation}
For $v \in \mcV_{\vartheta \rho}(\Lambda)$ and all $T \in \mcT$, this yields the interpolation estimate
\begin{align}
\label{eq:tensor-clement}
\norm{(\id - \mcI)v}_{L^2(\Gamma,\gamma;L^2(T))} &= \left( \int_\Gamma \Bigl\| \sum_{\mu \in \Lambda} (\id - I)v_\mu H_\mu(y) \Bigr\|_{L^2(T)}^2 \dgam{y} \right)^{1/2} \\
&= \left( \sum_{\mu \in \Lambda} \norm{(\id - I)v_\mu}_{L^2(T)}^2 \right)^{1/2} \\
&\leq c_{\mathcal T} h_T \left( \int_\Gamma \Bigl| \sum_{\mu \in \Lambda} v_\mu H_\mu(y) \Bigr|_{\mcX,\omega_T}^2 \dgam{y} \right)^{1/2} \\
&= c_{\mathcal T} h_T \abs{v}_{\mcV_{\vartheta \rho},\omega_T}. \\
\intertext{Likewise, on the edges $F \in \mcS$ we derive}
\norm{v- \mcI v}_{L^2(\Gamma,\gamma;L^2(F))} &\leq c_{\mathcal S} h_F^{1/2} \abs{v}_{\mcV_{\vartheta \rho},\omega_F}.
\end{align}
Here,
\begin{align}
\abs{v}_{\mcV_{\vartheta \rho},\omega_T}^2 := \int_\Gamma \abs{v(y)}_{\mcX,\omega_T}^2 \dgam{y}, \\ \abs{v}_{\mcV_{\vartheta \rho},\omega_F}^2 := \int_\Gamma \abs{v(y)}_{\mcX,\omega_F}^2 \dgam{y}.
\end{align}

\begin{theorem}[\cite{MR2805155} Thm. 2.45]
If $f\in L^p(\Gamma,\gamma_\rho;\mcX^\ast)$ for $p>2$, the Galerkin projection $u_N\in\mcV_N$ satisfies
\begin{equation}
\norm{u-u_N}_{L^2(\Gamma,\gamma;\mcX)} \leq \frac{\hat c_{\vartheta \rho}}{\check c_{\vartheta \rho}} \inf_{v_N\in\mcV_N(\Lambda;\mcT,p)} \norm{u-v_N}_{L^2(\Gamma,\gamma_\rho;\mcX)}.
\end{equation}
\end{theorem}

\begin{remark}
It should be noted that the constant $\frac{\hat c_{\vartheta \rho}}{\check c_{\vartheta \rho}}$ tends to $\infty$ as $\rho\to \{0,\infty\}$ and for $\theta\to\{0,1\}$, see Remark 2.46 in \cite{MR2805155}.
This is expected, as the problem is ill-posed in these limits.
In order to obtain reasonable upper error bounds, the parameters have hence to be chosen judiciously.
A more detailed investigation of an optimal parameter choice is postponed to future research. 
\end{remark}

\section{Decomposition of the Operator}
\label{sec:Tensor Decomposition Formats}

In this section, we introduce the discretization of the operator in an appropriate tensor format.
For this, an efficient representation of the non-linear coefficient is derived.
We first introduce basic aspects of the employed Tensor Train (TT) format.

\subsection{The Tensor Train format}
\label{sec:tt format}

We only provide a brief overview of the notation used regarding the tensor train representation.
For further details, we refer the reader to~\cite{EPS,Bachmayr2016} and the references therein.

Any function $w_N \in \mathcal V_N(\Lambda;\mathcal T,p)$ can be written as
\begin{equation}
w_N = \sum_{i=0}^{N-1} \sum_{\mu \in \Lambda} W(i,\mu) \phi_i \Hthrh_\mu.
\end{equation}
Thus, the discretization space is isomorphic to the tensor space, namely
\begin{equation}
\mathcal V_N(\Lambda;\mathcal T,p) \simeq \mathbb R^{N \times d_1 \times \cdots \times d_M}.
\end{equation}
The tensor $W$ grows exponentially with the order $M$, which constitutes the so called \emph{curse of dimensionality}.
We employ a low-rank decomposition of the tensor for a dimension reduction.
In this paper, we adhere to the \emph{Tensor Train (TT)} format for tensor decomposition~\cite{Oseledets2009}.
This seems reasonable, as the components (of the operator and hence the solution) are of decreasing importance due to the decay of the coefficient functions $b_m$ and therefore we can expect decreasing ranks in the tensor train format. Nevertheless, other tensor formats are also feasible in principle.

The TT representation of a tensor $W \in \mathbb R^{N \times d_1 \times \cdots \times d_M}$ is given as
\begin{equation}
W(i,\mu_1,\ldots,\mu_M) = \sum_{k_1=1}^{r_1} \cdots \sum_{k_M=1}^{r_M} W_0(i,k_1) \prod_{m=1}^M W_m(k_m,\mu_m,k_{m+1}).
\end{equation}
For simplicity of notation, we set $r_0 = r_{M+1} = 1$. If all dimensions $r_m$ are minimal, then this is called the \emph{TT decomposition} of $W$ and $\mathrm{rank}_{\mathrm{TT}}(W) := r = (1,r_1,\ldots,r_M,1)$ is called the \emph{TT rank} of $W$. The TT decomposition always exists and it can be computed in polynomial time using the \emph{hierarchical SVD} (HSVD)~\cite{Holtz2012b}. A truncated HSVD yields a \emph{quasi-optimal} approximation in the Frobenius norm~\cite{Oseledets2009,Grasedyck2010,Hackbusch2012,Hackbusch2014}. Most algebraic operations can be performed efficiently in the TT format~\cite{Bachmayr2016}.

Once the function $w_N$ has a low-rank TT decomposition, it is advisable to obtain a similar representation for the Galerkin operator on $\mathcal V_N(\Lambda;\mathcal T,p)$ in order to allow for efficient tensor solvers. For the lognormal coefficient $a(x,y)$, this can only be done approximately.

Later, it will be useful to express the storage complexity of a tensor train.
We distinguish the degrees of freedom given by the tensor train representation and the full (uncompressed) degrees of freedom.
For a tensor $U\in\mathbb{R}^{q_0\times\ldots\times q_L}$ of TT-rank $r=(1, r_1,\ldots,r_{L},1)$, the dimension of the low-rank tensor manifold is given by
\begin{equation}
  \label{eq:ttdofs}
  \text{tt-dofs}(U) := \sum_{\ell=1}^{L-1} (r_\ell q_\ell r_{\ell+1} - r^2_{\ell+1}) + r_{L}q_L,
\end{equation}
while the dimension of the full tensor space and hence its representation is
\begin{equation}
\text{full-dofs}(U) := \prod_{\ell=0}^L q_{\ell}.
\end{equation} 

One can conclude from~\eqref{eq:ttdofs} that the complexity of tensor trains depend only linear on the dimension, i.e., we have to store 
\begin{equation}
 \mathcal{O}(L\hat q \hat{r}^2), \quad \hat r = \max\left\{r\right\} \quad \hat{q} = \max\left\{q_0, \ldots, q_L\right\}
\end{equation}
 entries instead of $\mathcal{O}(\hat{q}^L)$, which is much smaller for moderate TT-ranks $r$.

\subsection{TT representation of the non-linear coefficient}
\label{eq:coefficient representation}

We approximate the coefficient
\begin{equation}
\label{eq:coef-prod}
a(x,y) = \exp\left(\sum_{\ell=1}^\infty b_\ell(x)y_\ell\right) = \prod_{\ell=1}^\infty e^{b_\ell(x)y_\ell}
\end{equation} 
using the coefficient splitting algorithm described in~\cite{MPDiss}. 
This results in a discretized coefficient on a tensor set 
\begin{equation}
\Delta = \{(\nu_1,\ldots,\nu_L,0,\ldots) \in \mathcal F : \nu_\ell = 0,\ldots, q_\ell-1, \ \ell = 1,\ldots,L \}
\end{equation}
with TT-rank $s = (1, s_1,\ldots,s_L, 1)$.
Here, we exploit Lemma \ref{lem:expH}, i.e., the fact that every factor of~\eqref{eq:coef-prod} has a Hermite expansion of the form
\begin{equation}
\label{eq:coeffherm}
\exp(b_\ell(x)y_\ell) = \sum_{\nu_\ell=0}^\infty c_{\nu_\ell}^{(\ell)}(x) \Hthrh_{\nu_\ell}(y_\ell)
\end{equation}
with
\begin{equation}
c_{\nu_\ell}^{(\ell)}(x) = \frac{(b_\ell(x)\sigma_\ell(\vartheta\rho))^{\nu_\ell}}{\sqrt{\nu_\ell !}} \exp((b_\ell(x)\sigma_\ell(\vartheta\rho))^2/2).
\end{equation}

The procedure is as follows: First, we fix an adequate quadrature rule for solving the involved integrals by choosing quadrature points $\chi_q \in D$ and weights $w_q \in \mathbb R$ for $q = 1,\ldots,P_{\mathrm{quad}}$.
We begin the discretization at the right most side and define the correlation matrix
\begin{align}
C_L(\nu_L,\nu_L') &:= \sum_{q=1}^{P_{\mathrm{quad}}} c_{\nu_L}^{(L)}(\chi_q) c_{\nu_L'}^{(L)}(\chi_q) w_q\\
&\approx \int_D c_{\nu_L}^{(L)}(x) c_{\nu_L'}^{(L)}(x) \dx{x}
\end{align}
for $\nu_L,\nu_L' = 0,\ldots,q_L-1$. This means that we have truncated the expansion~\eqref{eq:coeffherm} according to the tensor set $\Delta$, which yields an approximation of the factors.
This matrix is symmetric and positive semidefinite and it therefore admits an eigenvalue decomposition
\begin{equation}
C_L(\nu_L,\nu_L') = \sum_{k_L=1}^{q_L} \lambda_{k_L} A_L(k_L,\nu_L) A_L(k_L,\nu_L').
\end{equation}
This yields reduced basis functions
\begin{equation}
\tilde c_{k_L}^{(L)}(\chi_q) := \sum_{\nu_L=0}^{q_L-1} A_L(k_L,\nu_L) c_{\nu_L}^{(L)}(\chi_q)
\end{equation}
for $k_L = 1,\ldots,s_L$ that we can store explicitly at the quadrature points of the integral. If we choose $s_L = q_L$ then this is just a transformation without any reduction.

We proceed successively for $\ell = L-1,\dots,1$ by defining correlation matrices
\begin{equation}
C_\ell(\nu_\ell,k_{\ell+1},\nu_L',k_{\ell+1}') := \sum_{q=1}^{P_{\mathrm{quad}}} c_{\nu_\ell}^{(\ell)}(\chi_q) \tilde c_{k_{\ell+1}}^{(\ell+1)}(\chi_q) c_{\nu_\ell'}^{(\ell)}(\chi_q) \tilde c_{k_{\ell+1}'}^{(\ell+1)}(\chi_q) w_q
\end{equation}
with eigenvalue decompositions
\begin{equation}
C_\ell(\nu_\ell,k_{\ell+1},\nu_L',k_{\ell+1}') = \sum_{k_\ell=1}^{q_\ell} \lambda_{k_\ell} A_\ell(k_\ell,\nu_\ell,k_{\ell+1}) A_\ell(k_\ell,\nu_\ell',k_{\ell+1}')
\end{equation}
and the resulting reduced basis functions at the quadrature points
\begin{equation}
\tilde c_{k_\ell}{(\ell)}(\chi_p) = \sum_{\nu_\ell=0}^{q_\ell-1} \sum_{k_{\ell+1}=1}^{s_\ell} A_\ell(k_\ell,\nu_\ell,k_{\ell+1}) c_{\nu_\ell}^{(\ell)}(\chi_q) \tilde c_{k_{\ell+1}}^{(\ell+1)}(\chi_q),
\end{equation}
see Algorithm~\ref{alg:coeff-splitting}.
\begin{algorithm}[t]
\SetKwInOut{Input}{input}\SetKwInOut{Output}{output}
\Input{Coefficient functions $c_{\nu_1}^{(1)},\ldots,c_{\nu_L}^{(L)}, \quad \nu_\ell = 0,\ldots,q_\ell-1; \quad \ell=1,\ldots,L$; \\
ranks $s_1,\ldots,s_L$; \\
quadrature rule $(\chi_q,w_q), q=1,\ldots,P_{\mathrm{quad}}$.}
\Output{TT Tensor components $A_1,\ldots,A_L$;}
\BlankLine
Set $s_{L+1} = 1, \tilde c_1^{(L+1)} \equiv 1$\;
\For{$\ell \leftarrow L$ \DownTo $1$}{
	Arrange correlation matrix $C_\ell$:
	\begin{align}
	C_{\ell}(\nu_\ell,k_{\ell+1},\nu_\ell',k_{\ell+1}') &:= \sum_{q=1}^{P_{\mathrm{quad}}} c_{\nu_\ell}^{(\ell)}(\chi_q) \tilde c_{k_{\ell+1}}^{(\ell+1)}(\chi_q) c_{\nu_{\ell}'}^{(\ell)}(\chi_q) \tilde c_{k_{\ell+1}'}^{(\ell+1)}(\chi_q) w_q; \hspace{30pt} \\[-.5cm]
	\intertext{Compute eigenvalue decomposition:}
	C_\ell(\nu_\ell,k_{\ell+1},\nu_\ell',k_{\ell+1}') &= \sum_{k_\ell = 1}^{s_\ell} \lambda_{k_\ell} A_\ell(k_\ell,\nu_\ell,k_{\ell+1}) A_\ell(k_\ell,\nu_\ell',k_{\ell+1}');
	\end{align}
	Set reduced basis functions:
	$$\tilde c_{k_\ell}^\ell(\chi_q) = \sum_{\nu_\ell=0}^{q_\ell - 1} \sum_{k_{\ell+1}=1}^{s_{\ell+1}} A_\ell(k_\ell,\nu_\ell,k_{\ell+1}) c_{\nu_\ell}^\ell(\chi_q) \tilde c_{k_{\ell+1}}^{\ell+1}(\chi_q);$$
}
\caption{Algorithm for coefficient splitting.}
\label{alg:coeff-splitting}
\end{algorithm}
This results in a first component 
\begin{equation}
a_0[k_1](\chi_q) := \tilde c_{k_1}^{(1)}(\chi_q)
\end{equation}
for $k_1 = 1,\ldots,s_1$.
Note that on the one hand, it is possible to evaluate this component at any point $x \in D$ by converting the reduced basis functions back into their original form by means of the tensor components $A_\ell$ and the coefficient functions $c_{\nu_\ell}^{(\ell)}$.
More specifically,
\begin{align}
a_0[k_1](x) &= \sum_{\nu_1=0}^{q_1} \sum_{k_2=1}^{s_2} A_1(k_1,\nu_1,k_2) c_{\nu_1}^{(1)}(x) \tilde c_{k_2}^{(2)}(x) \\
&= \ldots \\
&= \sum_{k_2=1}^{s_2} \cdots \sum_{k_L=1}^{s_L} \prod_{\ell=1}^L \biggl( \sum_{\nu_\ell=0}^{q_\ell-1} A_\ell(k_\ell,\nu_\ell,k_\ell) c_{\nu_\ell}^{(\ell)}(x) \biggr). 
\end{align}
On the other hand, each original coefficient function is approximated by the reduced basis representation
\begin{align}
c_{\nu_L}^{(L)}(x) &\approx \sum_{k_L=1}^{s_L} A_L(k_L,\nu_L) \tilde c_{k_L}^{(L)}(x), \\
c_{\nu_\ell}^{(\ell)}(x) \tilde c_{k_{\ell+1}}^{(\ell+1)} &\approx \sum_{k_\ell=1}^{s_\ell} A_\ell(k_\ell,\nu_\ell,k_{\ell+1}) \tilde c_{k_\ell}^{(\ell)} \qquad\text{for all } \ell = L-1,\ldots,1.
\end{align}
This approximation is exact if the ranks $s = (s_1,\ldots,s_L)$ are full.

By the described procedure we obtain an approximate discretization $a_{\Delta,s} \approx a$ in a TT-like format that is continuous in the first component and that has the decomposition
\begin{equation}
\label{eq:tt-coef}
a_{\Delta,s}(x,y) = \sum_{k_1=1}^{s_1} \cdots \sum_{k_L=1}^{s_L} a_0[k_1](x) \left( \sum_{\nu \in \Delta} \prod_{\ell=1}^L A_\ell(k_\ell,\nu_\ell,k_{\ell+1}) \Hthrh_{\nu_\ell}(y_\ell) \right),
\end{equation}
see Figure~\ref{fig:coeff-splitting}.

\begin{figure}[t]
\centering
\includegraphics[width=\textwidth]{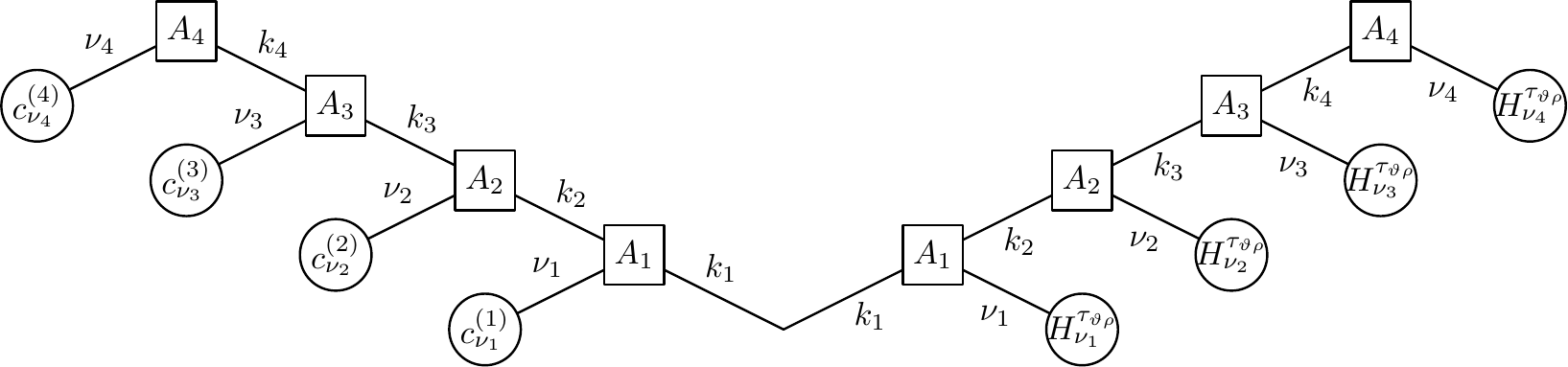}
\caption{A coefficient splitting for $L = 4$.}
\label{fig:coeff-splitting}
\end{figure}

On $\mathcal V_N(\Lambda;\mathcal T,p)$ the linear Galerkin operator $\mathbf A$ is in the TT matrix or matrix product operator (MPO) format:
\begin{equation}
\label{eq:operator}
\mathbf A(i,\mu,j,\mu') = \sum_{k_1=1}^{s_1} \cdots \sum_{k_M=1}^{s_M} \mathbf A_0(i,j,k_1) \prod_{m=1}^M \mathbf A_m(k_m,\mu_m,\mu_m',k_{m+1}),
\end{equation}
where
\begin{equation}
\mathbf A_0(i,j,k_1) = \int_D a_0[k_1](x) \nabla \phi_i \cdot \nabla \phi_j \dx{x}
\end{equation}
and for all $m = 1,\ldots,M-1$
\begin{align}
\mathbf A_m(k_m,\mu_m,\mu_m',k_{m+1}) &= \sum_{\nu_m=0}^{q_m-1} A_m(k_m,\nu_m,k_{m+1}) \\
& \qquad \qquad \times \int_{\mathbb R} \Hthrh_{\nu_m} \Hthrh_{\mu_m} \Hthrh_{\mu_m'} \dx{\gamma_{\theta\rho,m}}(y_m) \\
&= \sum_{\nu_m=0}^{q_m-1} A_m(k_m,\nu_m,k_{m+1}) \kappa_{\mu_m,\mu_m',\nu_m} \\
\end{align}
and
\begin{align}
\mathbf A_M(k_m,\mu_M,\mu_M') &= \sum_{\nu_M=0}^{q_M-1} \sum_{k_{m+1}=1}^{s_{m+1}} \cdots \sum_{k_L=1}^{s_L} A_M(k_m,\nu_M,k_{m+1}) \kappa_{\mu_M,\mu_M',\nu_M} \\
& \qquad \qquad \times \prod_{\ell=m+1}^L A_\ell(k_\ell,0,k_{\ell+1}) .
\end{align}
Since the integral over the triple product $\kappa_{\mu_m,\mu_m',\nu_m} = 0$ for all $\nu_m > 2\max(\mu_m,\mu_m')$, it is sufficient to set $q_\ell = 2d_\ell - 1$ for all $\ell = 1,\ldots,L$.
If the rank $s$ of the decomposition of the coefficient is full, then the discretised coefficient $a_{\Delta,s}$ is exact \emph{on the discrete space} $\mathcal V_N$ (up to quadrature errors).

However, this is generally infeasible as the rank would grow exponentially with $M$.
Therefore, a truncation of the rank becomes necessary and the coefficient is only an approximation.
We assume in the following that the error that is due to this approximation of the coefficient is small. A thorough estimation of this error is subject to future research.


\begin{remark}
A similar approach to decomposing the coefficient has been chosen in~\cite{Matthies2014} where the knowledge of the eigenfunctions of the covariance operator was assumed a priori.
This means that one has an orthogonal basis also for the deterministic part in $x$ and all that remains to do is to decompose the coefficient tensor for this basis representation.
This is also done using some quadrature and the $L^2$-error of this approximation can be estimated.
\end{remark}

According to the discretization of the coefficient, we introduce the splitting of the operator,
\begin{equation}
\label{eq:A-split}
\mathcal A(v) = \mathcal A_+(v) + \mathcal A_-(v),
\end{equation}
with the active and inactive operators $\mathcal A_+ : \mathcal V_N(\Lambda;\mathcal T,p) \rightarrow \mathcal V_N(\Lambda;\mathcal T,p)^*$ and $\mathcal A_- : \mathcal V_N(\Lambda;\mathcal T,p) \rightarrow \mathcal V_N(\Lambda;\mathcal T,p)^*$ defined by
\begin{align}
\label{eq:log-appr-ops}
\mathcal A_+(v) &:= - \ddiv (a_{\Delta,s} \nabla v), \\
\mathcal A_-(v) &:= - \ddiv \Bigl( (a - a_{\Delta,s}) \nabla v \Bigr),
\end{align}
for $v \in \mathcal V_N(\Lambda;\mathcal T,p)$.

The above considerations yield that the approximate operator $\mathcal A_+$ has the following tensor product structure
\begin{equation}
\label{eq:log-tens-op}
\mathcal A_+ = \sum_{k_1 =1}^{s_1} \cdots \sum_{k_\ell = 1}^{s_L} \mathcal A_0[k_1] \otimes \mathcal A_1[k_1,k_2] \otimes \cdots \otimes \mathcal A_L[k_L],
\end{equation}
with the operator components
\begin{align}
\mathcal A_0[k_1] : \mathcal X \rightarrow \mathcal X^*, \qquad \mathcal A_0[k_1](v_x) = - \ddiv (a_0[k_1] \nabla v_x),
\end{align}
and for all $\ell = 1,\ldots,L$,
\begin{align}
\mathcal A_\ell[k_\ell,k_{\ell + 1}] : L^2(\bbR,\gamma_{\vartheta \rho,m}) &\rightarrow L^2(\bbR,\gamma_{\vartheta \rho,m}) \\
\mathcal A_\ell[k_\ell,k_{\ell + 1}](v_y) &= \sum_{\nu_\ell=0}^{q_\ell-1} A_\ell(k_\ell,\nu_\ell,k_{\ell + 1}) \Hthrh_{\nu_\ell} v_y.
\end{align}

\section{Error Estimates}
\label{sec:Error Estimates}

This section is concerned with the derivation of a reliable a posteriori error estimator based on the stochastic residual.
In comparison to the derivation in~\cite{EGSZ,EGSZ2,EPS}, the lognormal coefficient requires a more involved approach directly related to the employed weighted function spaces introduced in Section~\ref{sec:Galerkin Discretization}.
In theory, an additional error occurs because of the discretization of the coefficient which we assume to be negligible.
The developed adaptive algorithm makes possible a computable a posteriori steering of the error components by a refinement of the FE mesh and the anisotropic Hermite polynomial chaos of the solution.
The efficient implementation is due to the formulation in the TT format the ranks of which are also set adaptively.

The definition of the operators as in~\eqref{eq:log-appr-ops} leads to a decomposition of the residual
\begin{equation}
\mathcal R(v) = \mathcal R_+(v) + \mathcal R_-(v),
\end{equation}
with
\begin{equation}
\mathcal R_+(v) := f - \mathcal A_+(v),\qquad
\mathcal R_-(v) := - \mathcal A_-(v).
\end{equation}
The discrete solution $w_N \in \mathcal V_N$ reads
\begin{equation}
w_N = \sum_{i=0}^{N-1} \sum_{\mu \in \Lambda} W(i,\mu) \varphi_i \Hthrh_\mu.
\end{equation}
We assume that the operator is given in its approximate semi-discrete form $\mathcal A_+$ and aim to estimate the energy error
\begin{equation}
\| u - w_N \|_{\mathcal A_+}^2 = \int_\Gamma \int_D a_{\Delta,s} \abs{\nabla (u - w_N)}^2 \, \dx{x} \dgamthrh{y}.
\end{equation}

\begin{remark}
As stated before, we assume that the error that results from approximating the coefficient is small. Estimation of this error is subject to future research. Work in this direction has e.g. been carried out in~\cite{MR2888311,Matthies2014}. 
Additionally, we require that the bounds~\eqref{eq:Vtr-boundedness} and~\eqref{eq:Vtr-coercivity} still hold, possibly with different constants $\hat c_{\vartheta \rho}^+$ and $\check c_{\vartheta \rho}^+$. 
This is for example guaranteed if $a_{\Delta,s}$ is positive, i.e., if 
\begin{equation}
a_{\Delta,s}(x,y) > 0 \qquad \forall x \in D, y \in \Gamma.
\end{equation}
Then, since the approximated coefficient is polynomial in $y$, the arguments in Lemma~\ref{lem:elliptic} yield the same constants
\begin{equation}
\hat c_{\vartheta \rho}^+ = \hat c_{\vartheta \rho}, \qquad \check c_{\vartheta \rho}^+ = \check c_{\vartheta \rho}.
\end{equation}
\end{remark}


We recall Theorem 5.1 from~\cite{EGSZ} and also provide the proof for the sake of a complete presentation.
Note that the result allows for non-orthogonal approximations $w_N\in\mcV_N$.
\begin{theorem}
\label{thm:mainError}
Let $\mcV_N\subset \mcV_{\vartheta\rho}$ a closed subspace and $w_N\in\mcV_N$, and let $u_N$ denote the $\mcA_+$ Galerkin projection of $u$ onto $\mcV_N$. 
Then it holds 
\begin{align}
\norm{u - w_N}_{\mcA_+}^2 &\leq \left(  \sup_{v \in \mcV_{\theta \rho} \setminus \{ 0 \}} \frac{\abs{ \langle \mathcal R_+(w_N), (\id - \mcI)v \rangle_{\theta \rho} }}{\check c_{\theta\rho}^+ \norm{v}_{L^2(\Gamma, \gamma;\mcX)}} + c_{\mathcal I} \norm{u_N - w_N}_{\mcA_+} \right)^2 \\
&\qquad \qquad \qquad + \norm{u_N - w_N}_{\mcA_+}^2.
\end{align}
\end{theorem}
Here, $\mathcal I$ denotes the Cl\'ement interpolation operator in~\eqref{eq:tensor-clement} and $c_{\mathcal I}$ is the operator norm of $\id - \mathcal I$ with respect to the energy norm $\norm{ \cdot }_{\mcA_+}$.
The constant $\check c_{\theta\rho}^+$ is derived from the assumed coercivity of the bilinear form induced by $\mcA_+$ similar to~\eqref{eq:Vtr-boundedness} and~\eqref{eq:Vtr-coercivity}. 

\begin{proof}
Due to Galerkin orthogonality of $u_N$, it holds
\begin{equation}
\norm{u-w_N}_{\mcA_+}^2 = \norm{u-u_N}_{\mcA_+}^2 + \norm{u_N-w_N}_{\mcA_+}^2.
\end{equation}
By the Riesz representation theorem, the first part is
\begin{equation}
\norm{u - u_N}_{\mcA_+} = \sup_{v \in \mcV_{\theta \rho} \setminus \{ 0 \}} \frac{\abs{\langle \mcR_+(u_N),v\rangle_{\theta \rho}}}{\| v \|_{\mcA_+}}.
\end{equation}
We now utilise the Galerkin orthogonality and introduce the bounded linear map $\mcI :  \mcV_{\theta \rho} \rightarrow \mcV_N$ to obtain
\begin{equation}
\norm{u - u_N}_{\mcA_+} = \sup_{v \in \mcV_{\theta \rho} \setminus \{ 0 \}} \frac{\abs{\langle \mcR_+(u_N), (\id - \mcI)v \rangle_{\theta \rho}}}{\| v \|_{\mcA_+}}.
\end{equation}
Since we do not have access to the Galerkin solution $u_N$, we reintroduce $w_N$
\begin{align}
\norm{u - u_N}_{\mcA_+} &\leq \sup_{v \in \mcV_{\theta \rho} \setminus \{ 0 \}} \frac{\abs{\langle \mcR_+(w_N), (\id - \mathcal I) v \rangle_{\theta \rho}}}{\| v \|_{\mcA_+}} \\
& \qquad \qquad + \frac{\abs{\langle \mcR_+(u_N) - \mcR_+(w_N), (\id - \mathcal I) v \rangle_{\theta \rho}}}{\| v \|_{\mcA_+}} \\
&\leq \sup_{v \in \mcV_{\theta \rho} \setminus \{ 0 \}} \frac{\abs{\langle \mcR_+(w_N), (\id - \mathcal I) v \rangle_{\theta \rho}}}{\| v \|_{\mcA_+}} \\
& \qquad \qquad + \frac{\norm{u_N - w_N}_{\mcA_+} \norm{(\id - \mathcal I) v}_{\mcA_+}}{\| v \|_{\mcA_+}} \\
&\leq \sup_{v \in \mcV_{\theta \rho} \setminus \{ 0 \}} \frac{\abs{\langle \mcR_+(w_N), (\id - \mathcal I) v \rangle_{\theta \rho}}}{\| v \|_{\mcA_+}} + c_{\mathcal I} \| w_N - u_N \|_{\mathcal A_+}.
\end{align}
We apply the coercivity of the operator $\mcA_+$ to the denominator, which yields the desired result.
For the last inequality, we used the boundedness of $\mathcal I$ in the energy norm by defining the constant as the operator norm
\begin{equation}
c_{\mathcal I} := \sup_{v \in \mcV_{\theta \rho} \setminus \{ 0 \}} \frac{\|(\id - \mathcal I) v\|_{\mcA_+}}{\| v \|_{\mcA_+}}.
\end{equation}
\end{proof}

Since the product of the Hermite polynomials for each $m = 1,\ldots,M$ has degree at most $q_m + d_m -2$, it is useful to define the index set
\begin{align}
 \Xi := \Delta + \Lambda &:= \bigl\{ \eta = (\eta_1,\ldots,\eta_L,0,\ldots) \colon \\
& \qquad \qquad \eta_m = 0,\ldots,q_m+d_m-2, \; m = 1,\ldots,M; \\
& \qquad \qquad \eta_\ell = 0,\ldots,q_\ell-1, \; \ell = M+1,\ldots,L \bigr\}.
\end{align}
Then, the residual can be split into an active and an inactive part by using the tensor sets $\Xi$ and $\Lambda$,
\begin{align}
\mathcal R_+(w_N) &= f - \mathcal A_+(w_N) \\
&= f + \sum_{\eta \in \Xi} \ddiv \res(\cdot,\eta) \Hthrh_\eta \\
&= \mathcal R_{+,\Lambda}(w_N) + \mathcal R_{+,\Xi \setminus \Lambda}(w_N),
\intertext{with}
\mathcal R_{+,\Lambda}(w_N) &= f + \sum_{\eta \in \Lambda} \ddiv \res(\cdot,\eta) \Hthrh_\eta, \\
\mathcal R_{+,\Xi \setminus \Lambda}(w_N) &= \sum_{\eta \in \Xi \setminus \Lambda} \ddiv \res(\cdot,\eta) \Hthrh_\eta,
\end{align}
where $\ddiv \res(\cdot,\eta) \in \mathcal X^*$ for all $\eta \in \Xi$.

For all $\eta \in \Xi$, the function $\res$ is given as
\begin{align}
\res(x,\eta) &= \sum_{k_1=1}^{r_1} \cdots \sum_{k_M=1}^{r_M} \sum_{k_1'=1}^{s_1} \cdots \sum_{k_L'=1}^{s_L} \res_0[k_1,k_1'](x) \\
&\quad \times \left( \prod_{m=1}^{M} R_m(k_m,k_m',\eta_m,k_{m+1},k_{m+1}') \prod_{\ell=M+1}^{L} A_\ell(k_\ell',\eta_\ell,k_{\ell+1}') \right)
\end{align}
with continuous first component 
\begin{equation}
\res_0[k_1,k_1'](x) = \sum_{i=0}^{N-1} a_0[k_1'](x) W_0(i,k_1) \nabla \varphi_i(x)
\end{equation}
and stochastic components for $m = 1,\ldots,M$,
\begin{align}
R_m(k_m,&k_m',\eta_m,k_{m+1},k_{m+1}') \\
&= \sum_{\mu_m = 0}^{q_m-1} \sum_{\nu_m = 0}^{d_m-1} A(k_m',\nu_m,k_{m+1}') W_m(k_m,\mu_m,k_{m+1}) \kappa_{\mu_m,\nu_m,\eta_m}.
\end{align}
The function $\res$ is again a TT tensor with continuous first component with TT ranks $r_ms_m$ for $m=1,\ldots,M$ and $s_\ell$ for $\ell = M+1,\ldots,L$. The physical dimensions are $d_m + q_m - 2$ for all $m = 1,\ldots,M$ and $d_\ell -1$ for $\ell = M+1,\ldots,L$.

The above considerations suggest that the error can be decomposed into errors that derive from the respective approximations in the deterministic domain, the parametric domain and in the ranks. This is indeed the case, as we will see in the following. In a nutshell, if $u_N$ is the Galerkin solution in $\mcV_N$ and $u_\Lambda$ is the Galerkin solution in the semi-discrete space $\mcV(\Lambda)$, then the \emph{deterministic error} $\err\Det = \norm{u_\Lambda - u_N}_{\mcA_+}$ corresponds to the error of the \emph{active residual} $\mathcal R_{+,\Lambda}$, the \emph{parametric error} $\err\param= \norm{u - u_\Lambda}_{\mcA_+}$ corresponds to the \emph{inactive residual} $\mathcal R_{+,\Xi \setminus \Lambda}$ and the error made by restricting the ranks is the error in the discrete space $\err\disc(w_N) = \norm{u_N - w_N}_{\mcA_+}^2$, see Figure~\ref{fig:errors} for an illustration.

\begin{figure}
\centering
\includegraphics{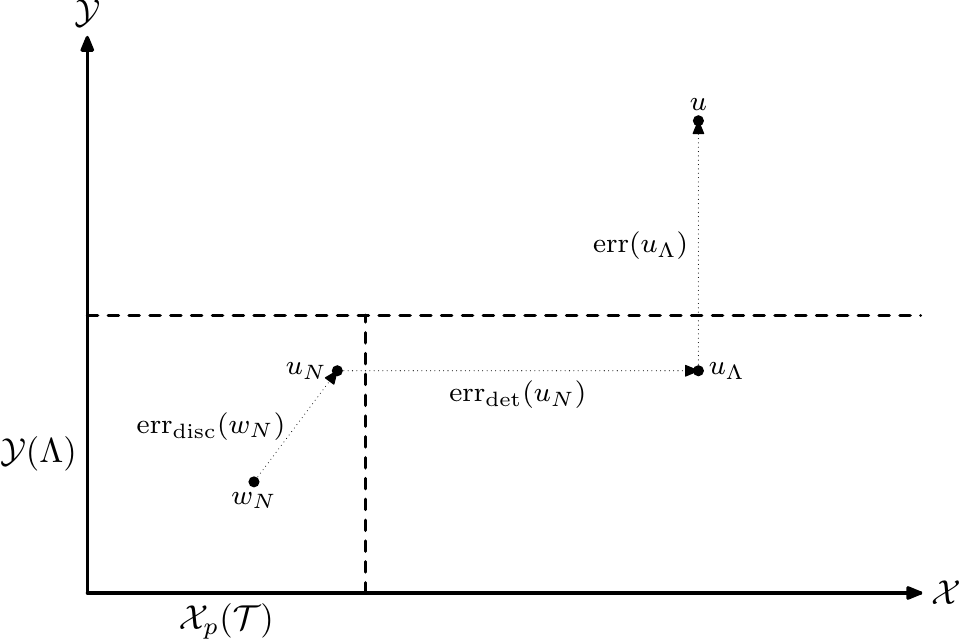}
\caption{An illustration of the different errors.}
\label{fig:errors}
\end{figure}

\subsection{Deterministic error estimation}
\label{sec:deterministic estimator}

We define the {\em deterministic error estimator}
\begin{align}
\est\Det(w_N) &:= \left( \sum_{T \in \mathcal T} \est_{\mathrm{det},T}^2(w_N) + \sum_{F \in \mathcal S} \est_{\mathrm{det},F}^2(w_N) \right)^{1/2}, \\
\est_{\mathrm{det},T}(w_N) &:= h_T \| \mathcal R_{+, \Lambda}(w_N) \zeta_{\vartheta \rho} \|_{L^2(\Gamma,\gamma;L^2(T))}, \\
\est_{\mathrm{det},F}(w_N) &:= h_F^{1/2} \| [\![ \mathcal R_{+, \Lambda}(w_N)]\!] \zeta_{\vartheta \rho} \|_{L^2(\Gamma,\gamma;L^2(F))}.
\end{align}
This estimates the active residual as follows.
\begin{proposition}
For any $v \in \mcV_{\vartheta\rho}$ and any $w_N \in \mcV_N$, it holds
\begin{equation}
\frac{\abs{\langle \mathcal R_{+,\Lambda}(w_N), (\id - \mcI)v \rangle_{\theta \rho}}}{\norm{v}_{L^2(\Gamma, \gamma;\mcX)}} \leq c\Det \est\Det(w_N).
\end{equation}
\end{proposition}
\begin{proof}
By localization to the elements of the triangulation $\mcT$ and integration by parts,
\begin{align}
\langle \mathcal R_{+,\Lambda}&(w_N), (\id - \mcI)v \rangle_{\theta \rho} \\ 
&= \int_\Gamma \sum_{T \in \mathcal T} \int_T f \bigl((\id - \mathcal I)v\bigr) - a_{\Delta,s} \nabla w_N \cdot \nabla \bigl((\id - \mathcal I)v\bigr) \, \dx{x} \dgamthrh{y} \\
&= \sum_{T \in \mathcal T} \int_\Gamma \int_T \mathcal R_{+,\Lambda}(w_N) \bigl((\id - \mathcal I)v\bigr) \, \zeta_{\vartheta \rho}(y) \dx{x} \dgam{y} \\
&\qquad + \sum_{F \in \mathcal S} \int_\Gamma \int_F [\![ \mathcal R_{+,\Lambda}(w_N) ]\!] \bigl((\id - \mathcal I)v\bigr) \zeta_{\vartheta \rho}(y) \dx{F}(x) \dgam{y}.
\end{align}
The Cauchy-Schwarz inequality yields 
\begin{align}
|\langle \mathcal R_{+,\Lambda}&(w_N), (\id - \mathcal I)v \rangle_{\theta\rho}| \\
&\leq \sum_{T \in \mathcal T} \| \mathcal R_{+,\Lambda}(w_N) \zeta_{\vartheta \rho} \|_{L^2(\Gamma,\gamma;L^2(T))} \|(\id - \mcI)v\|_{L^2(\Gamma,\gamma;L^2(T))} \\
&\qquad + \sum_{F \in \mathcal S} \| [\![ \mathcal R_{+,\Lambda}(w_N) ]\!] \zeta_{\vartheta \rho}\|_{L^2(\Gamma,\gamma;L^2(F))} \| (\id - \mcI)v \|_{L^2(\Gamma,\gamma;L^2(F))}.
\end{align}
With the interpolation properties~\eqref{eq:tensor-clement} we obtain
\begin{align}
|\langle \mathcal R_{+,\Lambda}&(w_N), (\id - \mathcal I)v \rangle_{\theta\rho} | \\
&\leq  \sum_{T \in \mathcal T} h_T c_{\mathcal T} \| \mathcal R_{+,\Lambda}(w_N) \zeta_{\vartheta \rho} \|_{L^2(\Gamma,\gamma;L^2(T))} | v |_{\mcV_{\vartheta\rho}, \omega_T} \\
&\qquad + \sum_{F \in \mathcal S} h_F^{1/2} c_{\mathcal{S}} \| [\![ \mathcal R_{+,\Lambda}(w_N) ]\!] \zeta_{\vartheta \rho} \|_{L^2(\Gamma,\gamma;L^2(F))} | v |_{\mcV_{\vartheta\rho}, \omega_F}.
\end{align}
Since the overlaps of the patches $\omega_T$ and $\omega_F$ are bounded uniformly, a Cauchy-Schwarz estimate leads to
\begin{equation}
|\langle \mathcal R_{+,\Lambda}(w_N), (\id - \mathcal I)v \rangle_{\theta\rho}| \leq c\Det \est\Det(w_N) \| v \|_{L^2(\Gamma, \gamma;\mcX)}.
\end{equation}
Here, the constant $c\Det$ depends on the properties of the interpolation operator~\eqref{eq:tensor-clement}.
\end{proof}

\begin{remark}
Note that an $L^2$-integration of the residual, which is an element of the dual space $\mcV_{\vartheta \rho}^*$, is possible since the solution consists of finite element functions.
These are piecewise polynomial and thus smooth on each element $T \in \mathcal T$. 
\end{remark}

\subsection{Tail error estimation}
\label{sec:tail-estimator}

The {\em parametric or tail estimator} is given by
\begin{equation}
\est\param(w_N) := \left( \int_\Gamma \int_D \Bigl( \sum_{\eta \in \Xi \setminus \Lambda} \res(x,\eta) \Hthrh_{\eta}(y) \, \zeta_{\vartheta\rho}(y) \Bigr)^2 \, \dx{x} \dgam{y} \right)^{1/2}
\end{equation}
and bounds the parametric error as follows.
\begin{proposition}
For any $v \in \mcV_{\vartheta\rho}$ and any $w_N \in \mcV_N$, it holds
\begin{equation}
\frac{\abs{\langle \mathcal R_{+,\Xi \setminus \Lambda}(w_N) , (\id - \mcI)v \rangle_{\theta \rho}}}{\norm{v}_{L^2(\Gamma, \gamma;\mcX)}} \leq \est\param(w_N).
\end{equation}
\end{proposition}
\begin{proof}
Recall that $\langle \mathcal R_{+,\Xi \setminus \Lambda}(w_N) , \mathcal I v \rangle_{\theta \rho} = 0$ since $\mathcal I v \in \mcV_N$.

Instead of factorizing out the $L^\infty$-norm of the diffusion coefficient as in~\cite{EGSZ,EGSZ2,EPS}, we use the Cauchy-Schwarz inequality to obtain
\begin{align}
\langle \mathcal R_{+, \Xi \setminus \Lambda}&(w_N), v \rangle_{\theta \rho} \\
&= \int_\Gamma \int_D \Bigl( \sum_{\eta \in \Xi \setminus \Lambda} \res(x,\eta) \Hthrh_\eta(y) \Bigr) \cdot \nabla v(x,y) \, \zeta_{\vartheta \rho}(y) \dx{x} \dgam{y} \\
&\leq \int_\Gamma \int_D \Bigl( \sum_{\eta \in \Xi \setminus \Lambda} \res(x,\eta) \Hthrh_\eta(y) \, \zeta_{\vartheta\rho}(y) \Bigr)^2 \dx{x} \dgam{y} \| v \|_{L^2(\Gamma, \gamma;\mcX)} \\
&= \est\param(w_N) \| v \|_{L^2(\Gamma, \gamma;\mcX)}.
\end{align}
\end{proof}

\subsection{Algebraic error estimation}
\label{sec:algebraic-estimator}

In order to define the {\em algebraic error estimator}, we need to state the linear basis change operator that translates integrals over two Hermite polynomials in the measure $\gamma_{\vartheta \rho}$ to the measure $\gamma$:
\begin{align}
\mathbf H_{\vartheta \rho \rightarrow 0} &: \mathbb R^{N \times d_1 \times \cdots \times d_M } \rightarrow \mathbb R^{N \times d_1 \times \cdots \times d_M}, \\
\mathbf H_{\vartheta \rho \rightarrow 0} &:= Z_0 \otimes Z_1 \otimes \cdots \otimes Z_M, \\
Z_0(i,j) &:= \int_D \nabla \varphi_i \cdot \nabla \varphi_j \dx{x}, \\
Z_m(\mu_m,\mu_m') &:= \int_{\bbR} \Hthrh_{\mu_m}(y_m) \Hthrh_{\mu_m'}(y_m) \, \dx{\gamma_m}(y_m) \quad \text{ for all } m=1,\ldots,M.
\end{align}
This yields the estimator
\begin{equation}
\label{eq:log-disc-est}
\est\disc(w_N) := \| (\mathbf A(W) - F) \mathbf H_{\vartheta \rho \rightarrow 0}^{-1/2} \|_{\ell^2(\mathbb R^{N \times d_1 \times \cdots \times d_M})}.
\end{equation}
\begin{proposition}
For any $w_N \in \mcV_N$ and the Galerkin solution $u_N \in \mcV_N$, it holds
\begin{equation}
\norm{u_N - w_N}_{\mcA_+} \leq (\check c_{\vartheta \rho}^+)^{-1} \est\disc(w_N).
\end{equation}
\end{proposition}
\begin{proof}
For $v_N = \sum_{i=0}^{N-1} \sum_{\mu \in \Lambda} V(i,\mu) \varphi_i \Hthrh_\mu \in \mathcal V_N$, it holds
\begin{equation}
\int_\Gamma \int_D \nabla v_N \cdot \nabla v_N \, \dx{x} \dgam{y} = \langle V \mathbf H_{\vartheta \rho \rightarrow 0}, V \rangle = \| V \mathbf H_{\vartheta \rho \rightarrow 0}^{1/2} \|_{\ell^2(\mathbb R^{N \times d_1 \times \cdots \times d_M})}.
\end{equation}
With this and using the coercivity of $\mcA_+$, we can see that
\begin{align}
\| w_N - u_N \|_{\mcA_+}^2 &= \int_\Gamma \int_D \mathcal A_+(w_N - u_N) \cdot \nabla (w_N - u_N) \, \dx{x} \dgamthrh{y} \\
&= \langle \mathbf AW - F , W - U \rangle \\
&= \langle (\mathbf AW - F) \mathbf H_{\vartheta \rho \rightarrow 0}^{-1/2} , (W - U) \mathbf H_{\vartheta \rho \rightarrow 0}^{1/2} \rangle \\
&\leq \|(\mathbf AW - F) \mathbf H_{\vartheta \rho \rightarrow 0}^{-1/2}\|_{\ell^2(\mathbb R^{N \times d_1 \times \cdots \times d_M})} \| w_N - u_N \|_{L^2(\Gamma,\gamma;\mathcal X)} \\
&\leq (\check c_{\vartheta \rho}^+)^{-1} \|(\mathbf AW - F) \mathbf H_{\vartheta \rho \rightarrow 0}^{-1/2}\|_{\ell^2(\mathbb R^{N \times d_1 \times \cdots \times d_M})} \| w_N - u_N \|_{\mathcal A_+}
\end{align}
and thus
\begin{equation}
\| w_N - u_N \|_{\mathcal A_+} \leq (\check c_{\vartheta \rho}^+)^{-1} \est\disc(w_N).
\end{equation}
\end{proof}

\subsection{Overall error estimation}
\label{sec:overall-estimator}

A combination of the above estimates yields an overall error estimator.
\begin{corollary}
\label{cor:full_error}
For any $w_N \in \mcV_N$, the energy error can be bounded by
\begin{align}
\norm{u - w_N}_{\mcA_+}^2 &\leq (\check c_{\vartheta \rho}^+)^{-2} \est_{\mathrm{all}}(w_N)^2
\intertext{with the error estimator given by}
\est_{\mathrm{all}}(w_N)^2 := \Bigl( c\Det \est\Det(w_N) &+ \est\param(w_N) + c_{\mcI} \est\disc(w_N) \Bigr)^2 \\ 
&+ \est\disc(w_N)^2.
\end{align}
\end{corollary}

\begin{remark}
In order to get suitable measures for the estimators, the squared density $\zeta_{\vartheta\rho}^2$ appears, which upon scaling with 
\begin{equation}
c_{\sigma} := \prod_{\ell=1}^L \frac{1}{\sigma_\ell \sqrt{2 - \sigma_\ell^2}}
\end{equation}
again is a Gaussian measure with standard deviation $\sigma' = (\sigma_\ell')_{1\leq \ell \leq L}$ for
\begin{equation}
\label{eq:new-sig}
\sigma_\ell' := \frac{\sigma_\ell}{\sqrt{2 - \sigma_\ell^2}}.
\end{equation} 


First, this adds a restriction on $\vartheta$ such that the argument in the square root is positive. This is fulfilled if $\exp(2 \vartheta \rho \alpha_1) < 2$, since $(\alpha_m)_{1\leq m \leq M}$ is a decreasing series, which can be ensured for some $\vartheta$ small enough.
Second, it is important to check whether the new measure is {\em weaker} or {\em stronger} than $\gamma_{\vartheta \rho}$ \cite{Gittelson2011a}, i.e., which space contains the other.
Since
\begin{equation}
\sigma_\ell' = \frac{\sigma_\ell}{\sqrt{2 - \sigma_\ell^2}} = \frac{\exp(\vartheta \rho \alpha_\ell)}{\sqrt{2 - \exp(2\vartheta \rho \alpha_\ell)}} \geq \exp(\vartheta \rho \alpha_\ell) = \sigma_\ell,
\end{equation}
functions that are integrable with respect to the measure $\gamma_{\vartheta \rho}$ are not necessarily integrable with respect to the squared measure.
However, since $f$ is independent of the parameters and $\mathcal A_+(w_N) \in \mathcal X^* \otimes \mathcal Y(\Xi)$ has a polynomial chaos expansion of finite degree, the residual $\mathcal R_+(w_N)$ is integrable over the parameters for any Gaussian measure and therefore it is also integrable with respect to the squared measure. 
\end{remark}

\subsection{Efficient computation of the different estimators}
\label{secc:efficient-computation}

The error estimators can be calculated efficiently in the TT format. 
For each element $T \in \mathcal T$ of the triangulation, the residual estimator is given by
\begin{align}
\est_{\mathrm{det},T}(w_N)^2 &= h_T^2 \| \mathcal R_{+, \Lambda}(w_N) \zeta_{\vartheta \rho} \|_{L^2(\Gamma,\gamma;L^2(T))}^2 \\
&= h_T^2 \int_\Gamma \int_T \Biggl(f + \sum_{\eta \in \Lambda} \ddiv \res(x,\eta) \Hthrh_\eta \Biggr)^2 \, \zeta_{\vartheta\rho}^2 \dx{x} \dgam{y} \\
& = h_T^2 (f,f)_{L^2(T)}\int_\Gamma \zeta_{\vartheta\rho}^2 \, \dgam{y} \\
& \qquad + 2 h_T^2 \sum_{\eta \in \Lambda} ( f, \ddiv \res(x,\eta))_{L^2(T)} \int_\Gamma \Hthrh_\eta \, \zeta_{\vartheta\rho}^2 \dgam{y} \\
&\qquad + \sum_{\eta\in\Lambda} \sum_{\eta'\in\Lambda} ( \ddiv \res(x,\eta), \ddiv \res(x,\eta') )_{L^2(T)} \\
&\qquad \qquad \times \int_\Gamma \Hthrh_\eta \Hthrh_{\eta'} \, \zeta_{\vartheta\rho}^2 \dgam{y}.
\end{align}
A complication of the change of the measure to $\gamma$ and the involved weight $\zeta_{\vartheta\rho}^2$ is the fact that the shifted Hermite polynomials $\Hthrh$ are not orthogonal with respect to this measure. 
However, this property can be restored easily by calculating the basis change integrals beforehand.
This results in another tensor product operator that is defined element-wise for $\eta,\eta' \in \Xi$ by
\begin{align}
\tilde{\mathbf H}(\eta,\eta') &:= \tilde Z_1(\eta_1,\eta_1') \cdots \tilde Z_L(\eta_L,\eta_L'), \\
\tilde Z_\ell(\eta_\ell,\eta_\ell') &:= \int_\Gamma \Hthrh_{\eta_\ell} \Hthrh_{\eta'_\ell} \, \zeta_{\vartheta\rho,\ell}^2 \dgam{y_\ell}.
\end{align}
This operator encodes the basis change to the squared measure and can be inserted in order to calculate the scalar product.
With this, the estimator takes the form
\begin{align}
\est_{\mathrm{det},T}(w_N)^2 &= h_T^2 (f,f)_{L^2(T)} \int_\Gamma \zeta_{\vartheta\rho}^2 \, \dgam{y} \\
&\qquad + 2 h_T^2 \sum_{\eta \in \Lambda} \tilde{\mathbf H}(\eta,0) ( f, \ddiv \res(x,\eta))_{L^2(T)} \\
&\qquad + \sum_{\eta\in\Lambda} \sum_{\eta'\in\Lambda} \tilde{\mathbf H}(\eta,\eta') ( \ddiv \res(x,\eta), \ddiv \res(x,\eta') )_{L^2(T)}.
\end{align}
Since $\tilde{\mathbf H}$ is a tensor product operator, this summation can be done component-wise, i.e.,~performing a matrix-vector multiplication of every component of the operator $\tilde{\mathbf H}$ with the corresponding component of the tensor function $r$.

Similarly, for the jump over the edge $F$ we obtain the estimator
\begin{equation}
\est_{\mathrm{det},F}(w_N)^2 = h_F \sum_{\eta\in\Lambda} \sum_{\eta'\in\Lambda} \tilde{\mathbf H}(\eta,\eta') ( [\![ \ddiv \res(x,\eta) ]\!] , [\![\ddiv \res(x,\eta') ]\!] )_{L^2(F)}.
\end{equation}
Analogously to the affine case dealt with in \cite{EPS}, both of these estimators can then be computed efficiently in the TT format.
The parametric error estimator $\est\param(w_N)$ can be estimated in a similar way.

To gain additional information about the residual influence of certain stochastic dimensions, we sum over specific index sets.
Let 
\begin{align}
\label{eq:Xin}
\Xi_m := \{ (\eta_1,\ldots,d_m,\ldots,\eta_M,0,\ldots) \in \mathcal F : \ &\eta_\ell = 0,\ldots,d_\ell-1, \\
&\ell=1,\ldots,\bcancel m,\ldots,M \},
\end{align}
where the strike through means that $\ell$ takes all values but $m$.
For every $m=1,2,\ldots$, and $w_N \in\mathcal V_N$ we define 
\begin{align}
\est_{\mathrm{param},m}(w_N)^2 &:= \int_\Gamma \int_D \zeta_{\vartheta\rho}^2 \Bigl| \sum_{\eta\in\Xi_m} \res(x,\eta) \Hthrh_\eta \Bigr|^2 \, \dx{x}\dgam{y}.
\end{align}
Using the same arguments and notation as above, we can simplify
\begin{align}
\est_{\mathrm{param},m}(w_N)^2 &= \int_D \sum_{\eta\in\Xi_n} \Bigl( \res(x,\eta) \cdot \sum_{\eta'\in\Xi_m} \tilde{\mathbf H}(\eta,\eta') \res(x,\eta') \Bigr) \, \dx{x}.
\end{align}
These operations, including the calculation of the discrete error estimator~\eqref{eq:log-disc-est}, can be executed efficiently in the TT format.

\section{Fully Adaptive Algorithm}
\label{sec:Fully Adaptive Algorithm}

With the derived error estimators of the preceding sections, it is possible to refine all discretization parameters accordingly.
As discussed before, the deterministic estimator assesses the error that arises from the finite element method.
The discrete error estimator evaluates the error made by a low rank approximation.
The rest of the error is estimated by the parametric error estimator.

The adaptive algorithm described in this section is similar to the algorithms presented in~\cite{EGSZ,EGSZ2,EPS}.
Given some mesh $\mathcal T$, a fixed polynomial degree $p$, a finite tensor set $\Lambda\subset\mathcal F$, and a start tensor $W$ with TT rank $r$, we assume that a numerical approximation $w_N \in \mathcal V_N$ is obtained by a function
\begin{equation}
w_N^+ \leftarrow \Solve[\Lambda,\mathcal T,r,W]
\end{equation}
where the superscript $+$ always denotes the updated object.
In our implementation, we used the preconditioned ALS algorithm but other solution algorithms are feasible as well.
The lognormal error indicators \ref{sec:Error Estimates} and thus the overall upper bound $\est_{\mathrm{all}}(w_N)$ in Corollary~\ref{cor:full_error} are computed by the methods
\begin{align}
(\est_{\mathrm{det},T}(w_N))_{T\in\mathcal T},\est\Det(w_N) &\leftarrow \Estimate_x[w_N,\Lambda,\mathcal T,p],\\
(\est_{\mathrm{param},m}(w_N))_{m\in\mathbb N}, \est\param(w_N) &\leftarrow \Estimate_y[w_N,\Lambda],\\
\est\disc(w_N) &\leftarrow \Estimate_{\mathrm{ALS}}[w_N].
\end{align}

\begin{algorithm}[t]
\SetKwInOut{Input}{input}\SetKwInOut{Output}{output}
\Input{Old solution $w_N$ with solution tensor $W$ and rank $r$; \\
mesh $\mathcal T$ with degrees $p$; \\
index set $\Lambda$;
accuracy $\epsilon_{\mathrm{TTASGFEM}}$.}
\Output{New solution $w_N^+$ with new solution tensor $W^+$; \\
new mesh $\mathcal T^+$, or new index set $\Lambda^+$, or new rank $r+1$.}
\BlankLine
$w_N^+ \leftarrow \Solve[\Lambda,\mathcal T,r,W]$\;
$(\est_{\mathrm{det},T})_{T\in\mathcal T},\est\Det \leftarrow \Estimate_x[w_N,\Lambda,\mathcal T,p]$\;
$(\est_{\mathrm{param},m})_{m\in\mathbb N}, \est\param \leftarrow \Estimate_y[w_N,\Lambda]$\;
$\est\disc \leftarrow \Estimate_{\mathrm{ALS}}[w_N]$\;
\While{$\est_{\mathrm{all}} > \epsilon_{\mathrm{TTASGFEM}}$}{
\If{$\est\Det = \max\{\est\Det,\est\param,\est\disc\}$}
{
  $\mathcal T_{\mathrm{mark}} \leftarrow \Mark_x[\ntheta, (\est_{\mathrm{det,T}})_{T\in\mathcal T}, \est\Det]$\;
  $\mathcal T^+ \leftarrow \Refine_x[\mathcal T,\mathcal T_{\mathrm{mark}}]$\;
}
\ElseIf{$\est\param = \max\{\est\Det,\est\param,\est\disc\}$}{
  $\mathcal I_{\mathrm{mark}} \leftarrow \Mark_y[(\est_{\mathrm{param},m})_{m\in\mathbb N}, \est\param]$\;
  $\Lambda^+ \leftarrow \Refine_y[\Lambda,\mathcal I_{\mathrm{mark}}]$\;
}
\Else{
$W^+ \leftarrow \Refine_{\mathrm{TT}}[W]$\;
}
$w_N^+ \leftarrow \Solve[\Lambda,\mathcal T,r,W^+]$\;
$(\est_{\mathrm{det},T})_{T\in\mathcal T},\est\Det \leftarrow \Estimate_x[w_N,\Lambda,\mathcal T,p]$\;
$(\est_{\mathrm{param},m})_{m\in\mathbb N}, \est\param \leftarrow \Estimate_y[w_N,\Lambda]$\;
$\est\disc \leftarrow \Estimate_{\mathrm{ALS}}[w_N]$\;
}
\caption{The TTASGFEM algorithm.}
\label{alg:asgfem}
\end{algorithm}

Depending on which error is largest, either the mesh is refined, or the index set $\Lambda$ is enlarged, or the rank $r$ of the solution is increased. This is done as follows:
\begin{itemize}
\item If the deterministic error $\est\Det(w_N)$ outweighs the others, we mark the elements $T \in \mathcal T$ that have the largest error $\est_{\mathrm{det},T}(w_N)$ until the sum of errors on the elements in this marked subset $\mathcal T_{\mathrm{mark}} \subset \mathcal T$ exceeds a certain ratio $0 < \ntheta < 1$. This is called the {\em D\"orfler marking strategy}
\begin{equation}
\sum_{T \in \mathcal T_{\mathrm{mark}}} \est_{\mathrm{det},T}(w_N) \geq \ntheta \est\Det(w_N).
\end{equation}
We denote this procedure by
\begin{equation}
\mathcal T_{\mathrm{mark}} \leftarrow \Mark_x[\ntheta,(\est_{\mathrm{det},T}(w_N,\Lambda))_{T\in\mathcal T}, \est\Det(w_N,\Lambda,\mathcal T)].
\end{equation}
The elements in this subset are subsequently refined by 
\begin{equation}
\mathcal T^+ \leftarrow \Refine_x[\mathcal T, \mathcal T_{\mathrm{mark}}].
\end{equation}
\item In case the parametric error $\est\param(w_N)$ dominates, we use estimators $\est_{\mathrm{param},m}(w_N)$ in order to determine which components need to be refined. Here, we also mark until the D\"orfler property is satisfied, that is, we obtain a subset $\mathcal I_{\mathrm{mark}} \subset \mathbb N$ such that
\begin{equation}
\sum_{m \in \mathcal I_{\mathrm{mark}}} \est_{\mathrm{param},m}(w_N) \geq \ntheta \est\param(w_N).
\end{equation}
This is the marking
\begin{equation}
\mathcal I_{\mathrm{mark}} \leftarrow \Mark_y[\ntheta,(\est_{\mathrm{param},m}(w_N))_{m\in\mathbb N}, \est\param(w_N)],
\end{equation}
and we refine by increasing $d_m^+ \leftarrow d_m + 1$ for $m \in \mathcal I_{\mathrm{mark}}$
\begin{equation}
\Lambda^+ \leftarrow \Refine_y[\Lambda,\mathcal I_{\mathrm{mark}}].
\end{equation}
\item Finally, if $\est\disc(w_N)$ exceeds the other errors, we simply add a random tensor of rank 1 to the solution tensor $W$.
This increases all TT ranks of $W$ by 1 almost surely.
It would also be possible to add an approximation of the discrete residual as this is also in TT format.
However, since the ALS algorithm will be performed after the refinement, the advantage of this rather costly improvement has shown to be negligible~\cite{Uschmajew2014}.
Thus we get
\begin{equation}
W^+ \leftarrow \Refine_{\mathrm{TT}}[W].
\end{equation}
\end{itemize}

A single iteration step of the adaptive algorithm returns either a refined $\mathcal T^+$ or $\Lambda^+$ or the tensor format solution with increased rank $W^+$ and then solves the problem with these properties. This is done repeatedly until the overall error estimator $\err_{\mathrm{all}}(w_N)$ is sufficiently small, i.e.,~defined error bound $\epsilon_{\mathrm{TTASGFEM}}$ or a maximum problem size is reached. This procedure is given by the function $\mathrm{TTASGFEM}$, displayed in Algorithm~\ref{alg:asgfem}.
The upper error bounds directly follow from Corollary~\ref{cor:full_error}. 

\section{Numerical Experiments}
\label{sec:Numerical Experiments}

In this section we examine the performance of the proposed adaptive algorithm with some benchmark problems.
The computations are carried out with the open source framework \texttt{ALEA}~\cite{ez:alea}.
The FE discretization is based on the open source package \texttt{FEniCS}~\cite{fenics}.
The shown experiments are similar to the ones of the predecessor paper~\cite{EPS} in Section 7.
As before, the model equation~\eqref{eq:model} is computed for different lognormal coefficients on the square domain.
The derived error estimator is used as a refinement indicator.
Of particular interest is the observed convergence of the true (sampled) expected energy error and the behaviour of the error indicator.
Moreover, we comment on the complexity of the coefficient discretization.

\subsection{Evaluation of the Error}
\label{sec:evaluation of the error}

The real error of the computed approximation is determined by a Monte Carlo estimation.
For this, a set of $\Nmc$ independent samples $\left(y^{(i)}\right)_{i=1}^\Nmc$ of the stochastic parameter vector is considered.
By the tensor structure of the probability measure $\gamma=\bigotimes_{m\in\bbN} \gamma_1$, each entry of the vector valued sample $y^{(i)}$ is drawn with respect to $\mathcal{N}_1$.
The parametric solution $w_N\in\mathcal{V}_N(\Lambda; \mcT, p)$ obtained by the adaptive algorithm at a sample $y^{(i)}$, is compared to the corresponding (deterministic) sampled solution $u(y^{(i)})\in\mcX_{p'}(\widetilde\mcT)$ on a much finer finite element space subject to a reference triangulation $\widetilde\mcT$, obtained from uniform refinements of the finest adaptively created mesh, up to the point where we obtain at least $10^5$ degrees of freedom with Lagrange elements of order $p'=3$.
For computations, the truncation parameter applied to the infinite sum in~\eqref{eq:a-param}, controlling the length of the affine field representation, is set to $M=100$.
The mean squared error is determined by a Monte-Carlo quadrature,
\begin{equation}
\label{eq:errornorm}
\int_\Gamma\norm{u(y) - w_N(y)}^2_{\mcX} \mathrm{d}\gamma(y) \approx \frac{1}{\Nmc}\sum_{i=1}^\Nmc \norm{u(y^{(i)}) - w_N(y^{(i)})}^2_{\mcX_p(\widetilde\mcT)}.
\end{equation}
A choice of $\Nmc=250$ proved to be sufficient to obtain a reliable estimate in our experiments.

\subsection{The stochastic model problem}
\label{sec:stochastic diffusion problem}

In the numerical experiments, we consider the stationary diffusion problem~\eqref{eq:model} on the square domain $D=(0,1)^2$.
As in \cite{EPS, EGSZ, EGSZ2, EM}, the expansion coefficients of the stochastic field~\eqref{eq:a-param} are given by
\begin{equation}
b_m(x) \defeq \gamma m^{-\sigma} \cos(2\pi \rho_1(m) x_1)\cos(2\pi \rho_2(m) x_2), \text{ for } m\geq 1.
\label{eq:experiments-coeff-form}
\end{equation}
We chose the scaling $\gamma=0.9$.

Moreover,
\begin{equation}
\label{eq:beta12}
\rho_1(m) = m - k(m)(k(m)+1)/2 \quad\text{ and }\quad \rho_2(m) = k(m) - \rho_1(m),
\end{equation}
with $k(m) = \lfloor -1/2 + \sqrt{1/4+2m} \rfloor$, \ie, the coefficient functions $b_m$ enumerate all planar Fourier sine modes in increasing total order.
To illustrate the influence which the stochastic coefficient plays in the adaptive algorithm, we examine the expansion with varying decay, setting $\sigma$ in~\eqref{eq:experiments-coeff-form} to different values.
The computations are carried out with conforming FE spaces of polynomial degree $1$ and $3$.

The fully discretized problem is solved in the TT format using the \emph{Alternating Least Squares (ALS)} algorithm as introduced in \cite{EPS}. Other algorithms like Riemannian optimization are also feasible~\cite{Absil2008,Kressner2014}.
The ALS has a termination threshold of $10^{-12}$ that needs to be reached in the Frobenius norm of the difference of two successive iteration results. 

For our choice of the coefficient field, the introduced weights in~\eqref{eq:Btr} are set to $\rho=1$ and $\theta=0.1$.

\subsection{Tensor train representation of the coefficient}
\label{sec:tt-coef-numeric}

Since the tensor approximation of the coefficient is the starting point for the discretization of the operator, we begin with an examination of the rank dependence of the coefficient approximation scheme given in Algorithm~\ref{alg:coeff-splitting}.
For this, we fix the multi-index set such that the incorporated Hermite polynomials are of degree $15$ in each stochastic dimension, i.e.,
\begin{equation}
\Delta = \{\nu\in\mathcal{F}\; :\; \nu_\ell = 0, \ldots, 16, \; \ell=1,\ldots, L\}.
\end{equation}
As a benchmark, we chose a slow decay with $\sigma=2$ and set the quadrature rule\footnote{the statements are with regard to the respective FE discretization in \texttt{FEniCS}} to exactly integrate polynomials of degree 7 with a grid such that the number of quadrature nodes is at least $P_{\mathrm{quad}} = 10^4$.
In the following, the relative root mean squared (RRMS) error
\begin{equation}
\mathcal{E}(a, a_{\Delta, s}) := \left(\mathbb{E}[\norm{a - a_{\Delta, s}}^2_{L^2(D)} / \norm{a}_{L^2(D)}^2]\right)^{1/2}
\end{equation}
is compared to the growth in degrees of freedom with respect to various tensor ranks.
Numerically, this expression is computed by a Monte Carlo approximation as described in Section~\ref{sec:evaluation of the error} with the reference mesh $\widetilde{\mathcal{T}}$.

By denoting $A=(A_0,A_1,\ldots, A_L)$ the component tensors of $a_{\Delta, s}$ in~\eqref{eq:tt-coef}, where $A_0[x_k, k_1] = a_0[k_1](x_k)$ corresponds to the evaluation of $a_0[k_1]$ at every node $x_k$ of $\widetilde{\mathcal{T}}$, we can apply the notion of tt-dofs~\eqref{eq:ttdofs} to the discrete tensor train coefficient.
To highlight the crucial choice of the rank parameter $s=(1, s_1,\ldots, s_L, 1) $ in $a_{\Delta, s}$, we let the maximal attainable rank $s_{\textrm{max}}$ vary.

\begin{figure}
  \def\arraystretch{1.3}
  \begin{tabular}{@{}llllll@{}}
  \toprule
    & $s_{\textrm{max}}$ & $\mathcal{E}(a, a_{\Delta, s})$ & tt-dofs($A$) & full-dofs($A$) & CPU-time \\
    \midrule
    
  $L=10$ \\
  & $10$ & $1.21\times 10^{-3}$ & $2.04\times 10^6$ & \multirow{4}{*}{$3.02\times 10^{17}$} & $1.7$ s \\
  & $20$ & $6.25\times 10^{-4}$ & $4.12\times 10^6$ & & $4.56$ s \\
  & $50$ & $6.18\times 10^{-4}$ & $1.05\times 10^7$ & & $16.52$ s \\
  & $100$ & $6.18\times 10^{-4}$ & $2.15\times 10^7$ & & $70.29$ s \\
  
  $L=50$ \\
    
  & $10$ & $1.17\times 10^{-3}$ & $2.11\times 10^6$ & \multirow{4}{*}{$3.26\times 10^{65}$} & $9.14$ s \\
  & $20$ & $3.26\times 10^{-4}$ & $4.36\times 10^6$ & & $13.56$ s \\
  & $50$ & $5.98\times 10^{-5}$ & $1.20\times 10^7$ & & $50.60$ s \\
  & $100$ & $5.16\times 10^{-5}$ & $2.75\times 10^7$ & & $247.63$ s \\

  $L=100$ \\
  
  & $10$  & $1.16\times 10^{-3}$ & $2.18\times 10^6$ & \multirow{4}{*}{$5.25\times 10^{125}$} & $19.09$ s \\
  & $20$  & $3.20\times 10^{-4}$ & $4.66\times 10^6$ & & $27.16$ s \\
  & $50$  & $6.33\times 10^{-5}$ & $1.38\times 10^7$ & & $100.53$ s \\
  & $100$  & $7.47\times 10^{-6}$ & $3.50\times 10^7$ & & $498.47$ s \\
  \bottomrule
  \end{tabular}
  
  \caption{Comparison of different coefficient field approximations with fixed stochastic parameter set of order $L=\{10, 50, 100\}$, fixed stochastic polynomial degree $15$ and different rank configurations having maximal rank $s_{\textrm{max}} = \{10, 20, 50, 100\}$}
  \label{fig:coef-err}
\end{figure}

It can be seen in Figure~\ref{fig:coef-err} that a higher rank yields a more accurate approximation of the coefficient, as long as the stochastic expansion $L$ is long enough.
For small numbers of $L$, the RRMS does not decrease any further than $6\times 10^{-4}$, even when increasing the maximal attainable rank.
Otherwise, for up to $L=100$ terms in the affine coefficient field parametrisation, a small rank parameter of $s_{\mathrm{max}}=10$ gives a reasonable RRMS error of $1\times 10^{-3}$ which can be improved by increasing $s_{\mathrm{max}}$.

It should be pointed out that the used approximation only becomes feasible computationally by the employed tensor format since without low-rank representation the amount of degrees of freedom grows exponentially in der number of expansion dimensions.
Furthermore, since the tensor complexity and arithmetic operations depend only linearly on the number of expansion dimensions, the computational time scales linearly as illustrated in the last column of Figure~\ref{fig:coef-err} showing the average run time of 10 runs\footnote{run on an 8GB Intel Core i7-6500 laptop}.

\subsection{Adaptivity in physical space}
\label{sec:adaptivity in physical space}

As a first example for an adaptive discretization, we examine the automatic refinement of the physical FE space only.
For this, the stochastic expansion of the coefficient is fixed to some small value $L=5$ in~\eqref{eq:tt-coef}, which also is assumed for the corresponding reference solution.
The considered polynomial (Hermite chaos) degree of the approximation in each dimension is fixed to $d_1 = \ldots = d_5 = 10$, which can be considered sufficiently large for the respective problem, given an algebraic decay rate of $\sigma=2$ in the coefficient.

Although this experiment does not illustrate the performance of the complete algorithm, it nevertheless depicts the varying convergence rates for different polynomial orders in the FE approximation.
The rank of the tensor train approximation is fixed to $r_{\mathrm{max}}=10$, which is sufficient due to the low stochastic dimensions.
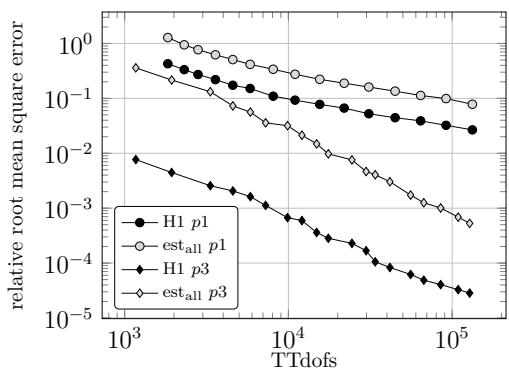
\begin{figure}[h]
  \begin{tikzpicture}[tikzpic options]
	\begin{loglogaxis}[sol error plot]
\pgfplotstableread[col sep=comma]{results/20180531022458sample/adaptive_P3_domsquare_sol_error.dat}\dataPPP
\pgfplotstableread[col sep=comma]{results/20180403175058sample/thx0.3_thy0.0_adaptive_P2_domsquare_scale1.0_ezw0.5_rzw20.0_sol_error.dat}\dataPP
\pgfplotstableread[col sep=comma]{results/20180403181243sample/thx0.3_thy0.0_adaptive_P1_domsquare_scale1.0_ezw0.5_rzw20.0_sol_error.dat}\dataP

\addplot table[x=tt_dof,y expr=\thisrow{H1_rel}] {\dataP}; \addlegendentry{H1 $p1$};
\addplot table[x=tt_dof,y expr=(((3*\thisrow{eta}+\thisrow{zeta}+\thisrow{resnorm})^2 + \thisrow{resnorm}^2)^0.5)] {\dataP}; \addlegendentry{$\est_{\mathrm{all}}$ $p1$};
\addplot table[x=tt_dof,y expr=\thisrow{rms_H1_rel}] {\dataPPP}; \addlegendentry{H1 $p3$};
\addplot table[x=tt_dof,y expr=(((3*\thisrow{eta}+\thisrow{zeta}+\thisrow{resnorm})^2 + \thisrow{resnorm}^2)^0.5)] {\dataPPP}; \addlegendentry{$\est_{\mathrm{all}}$ $p3$};
	\end{loglogaxis}
  \end{tikzpicture}
  \caption{Relative sampled mean squared $H^1(D)$ error for fixed stochastic dimension $L=5$ and adaptive refinement of the physical mesh.
  Each setting is shown along its respective overall error estimator as defined in Corollary~\ref{cor:full_error} and plotted against the total number of degrees of freedom in the TT representation.
  Considered are FE approximations of order $p=1$ and $p=3$.}
  \label{fig:mesh-only}
\end{figure}

It can be observed in Figure~\ref{fig:mesh-only} that the error estimator follows the rate of convergence of the sampled error.
Moreover, the higher-order FE method exhibits a higher convergence rate as expected.
This could also be observed in the (much simpler) affine case scrutinized in the preceding works~\cite{EPS,EGSZ,EGSZ2,EM}.

\subsection{Fully adaptive algorithm}
\label{sec:fully adaptive algorithm}

The fully adaptive algorithm described in Algorithm~\ref{alg:asgfem} is instantiated with a small initial tensor with full tensor rank $r_1=2$ consisting of a single $M=1$ stochastic component discretized with a linear polynomial $d_1=2$ and a physical mesh with $\abs{\mcT}=32$ elements for linear ansatz functions $p=1$ and $\abs{\mcT}=8$ for $p=3$.
The marking parameter is set to $\ntheta=0.5$.

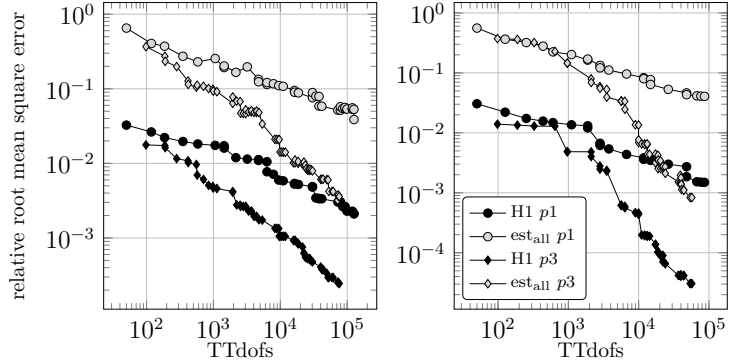
\begin{figure}[h]
  \begin{tikzpicture}[tikzpic options]
	\begin{loglogaxis}[sol error plot dbl 1]
\pgfplotstableread[col sep=comma]{results/20180524032340sample/adaptive_P3_domsquare_sol_error.dat}\dataPPP
\pgfplotstableread[col sep=comma]{results/20180525045915sample/adaptive_P1_domsquare_sol_error.dat}\dataP

\addplot table[x=tt_dof,y expr=(\thisrow{rms_H1_rel})] {\dataP}; \addlegendentry{H1 $p1$};
\addplot table[x=tt_dof,y expr=(((\thisrow{eta}+\thisrow{zeta}+\thisrow{resnorm})^2 + \thisrow{resnorm}^2)^0.5)^1] {\dataP}; \addlegendentry{$\est_{\mathrm{all}}$ $p1$};
\addplot table[x=tt_dof,y expr=(\thisrow{rms_H1_rel})] {\dataPPP}; \addlegendentry{H1 $p3$};
\addplot table[x=tt_dof,y expr=(((\thisrow{eta}+\thisrow{zeta}+\thisrow{resnorm})^2 + \thisrow{resnorm}^2)^0.5)^1] {\dataPPP}; \addlegendentry{$\est_{\mathrm{all}}$ $p3$};
\legend{}
	\end{loglogaxis}
  \end{tikzpicture}
  \begin{tikzpicture}[tikzpic options]
	\begin{loglogaxis}[sol error plot dbl 2]
\pgfplotstableread[col sep=comma]{results/20180525155407sample/adaptive_P3_domsquare_sol_error.dat}\dataPPP
\pgfplotstableread[col sep=comma]{results/20180524221112sample/adaptive_P1_domsquare_sol_error.dat}\dataP

\addplot table[x=tt_dof,y expr=(\thisrow{rms_H1_rel})] {\dataP}; \addlegendentry{H1 $p1$};
\addplot table[x=tt_dof,y expr=(((\thisrow{eta}+\thisrow{zeta}+\thisrow{resnorm})^2 + \thisrow{resnorm}^2)^0.5)] {\dataP}; \addlegendentry{$\est_{\mathrm{all}}$ $p1$};
\addplot table[x=tt_dof,y expr=(\thisrow{rms_H1_rel})] {\dataPPP}; \addlegendentry{H1 $p3$};
\addplot table[x=tt_dof,y expr=(((\thisrow{eta}+\thisrow{zeta}+\thisrow{resnorm})^2 + \thisrow{resnorm}^2)^0.5)] {\dataPPP}; \addlegendentry{$\est_{\mathrm{all}}$ $p3$};

	\end{loglogaxis}
  \end{tikzpicture}
  \caption{Relative, sampled, mean squared $H^1(D)$ error for the fully adaptive refinement. Each setting is shown along its respective overall error estimator as defined in Corollary~\ref{cor:full_error} and plotted against the total number of degrees of freedom in the representation. Considered are finite elements approximation of order $p=1$ and $p=3$ and slow decay ($\sigma=2$, left) and fast decay ($\sigma=4$, right). }  
  \label{fig:overall}
\end{figure}

Figure~\ref{fig:overall} depicts the real (sampled) mean squared energy error and the corresponding overall error estimator for FE discretizations of degrees $p=1,3$.
On the left-hand side of the figure, a lower decay rate ($\sigma=2$) in the coefficient representation is used, resulting in more stochastic modes to be relevant for an accurate approximation.
The right-hand side shows results for a faster decay rate with $\sigma=4$.
As expected, the convergence rate for the $p=3$ FE discretizations is faster than with $p=1$ FE.
Moreover, for the simpler problem with fast decay, we achieve a smaller error with a comparable number of degrees of freedom than in the harder slow decay setting.

\begin{remark}
Note that we neglect the (large) factor $\check c^+_{\theta\rho}$ in Corollary~\ref{cor:full_error}, which depends on the choice of weights $\theta$ and $\rho$ of the discrete space.
A detailed analysis of how to optimally select these weights is outside the scope of this article.
\end{remark}

\begin{figure}
  \def\arraystretch{1.3}
  \begin{tabular}{@{}lllllll@{}}

  \multicolumn{7}{c}{finite element order $p=1$ and slow coefficient decay $\sigma=2$} \\
  \toprule
    Iteration & $M$ & $d^{\textrm{max}}$ & $r^{\mathrm{max}}$ & m-dofs($W_N$) & tt-dofs($W_N$) & op-dofs($W_N$)  \\
  \cmidrule{1-7}
    
  5/37  & $1$ & $1$ & $2$ & $292$ & $584$ & $873$ \\
  15/37 & $2$ & $2$ & $5$ & $1577$ & $7900$ & $73965$ \\
  20/37 & $3$ & $4$ & $8$ & $2330$ & $18656$ & $170247$ \\
  30/37 & $5$ & $4$ & $13$ & $6586$ & $85847$ & $678382$ \\
  37/37 & $6$ & $5$ & $19$ & $6586$ & $126330$ & $734880$  \\

    \multicolumn{7}{c}{finite element order $p=1$ and fast coefficient decay $\sigma=4$} \\
  \toprule
    
  5/22  & $1$ & $1$ & $2$ & $302$ & $604$ & $1200$ \\
  10/22 & $1$ & $3$ & $3$ & $941$ & $2826$ & $11196$ \\
  15/22 & $1$ & $4$ & $5$ & $2951$ & $14755$ & $44115$ \\
  22/22 & $2$ & $5$ & $9$ & $9608$ & $86499$ & $962022$ \\
  \bottomrule
  \end{tabular}
  \caption{Stochastic expansion length, maximal polynomial chaos degrees, tensor ranks and degrees of freedom for the physical and (compressed) stochastic discretizations as well as for the operator in TT format for selected iteration steps in the fully adaptive algorithm.
  Finite elements of order $p=1$ are used to solve the problem with coefficient decay rates of $\sigma\in\{2, 4\}$.}
  \label{fig:adapt-process-p1}
\end{figure}

\begin{figure}
  \def\arraystretch{1.3}
  \begin{tabular}{@{}lllllll@{}}

  \multicolumn{7}{c}{finite element order $p=3$ and slow coefficient decay $\sigma=2$} \\
  \toprule
    Iteration & $M$ & $d^{\textrm{max}}$ & $r^{\mathrm{max}}$ & m-dofs($W_N$) & tt-dofs($W_N$) & op-dofs($W_N$)  \\
  \cmidrule{1-7}
    
  5/65  & $1$ & $2$ & $2$ & $139$ & $417$ & $816$ \\
  20/65 & $5$ & $3$ & $12$ & $241$ & $3040$ & $37537$ \\
  35/65 & $7$ & $6$ & $21$ & $403$ & $9856$ & $130799$ \\
  50/65 & $9$ & $11$ & $33$ & $568$ & $27662$ & $239273$ \\
  65/65 & $16$ & $11$ & $46$ & $1078$ & $76173$ & $353153$  \\

    \multicolumn{7}{c}{finite element order $p=3$ and fast coefficient decay $\sigma=4$} \\
    \cmidrule{1-7}
    
  5/67  & $1$ & $2$ & $2$ & $322$ & $966$ & $2226$ \\
  20/67 & $2$ & $7$ & $7$ & $991$ & $9974$ & $101110$ \\
  35/67 & $4$ & $13$ & $15$ & $1207$ & $20423$ & $250913$ \\
  50/67 & $6$ & $21$ & $19$ & $1909$ & $39998$ & $434669$ \\
  67/67 & $6$ & $34$ & $22$ & $2404$ & $61920$ & $607076$ \\
  \bottomrule

  \end{tabular}
  \caption{Stochastic expansion length, maximal polynomial chaos degrees, tensor ranks and degrees of freedom for the physical and (compressed) stochastic discretizations as well as for the operator in TT format for selected iteration steps in the fully adaptive algorithm.
  Finite elements of order $p=3$ are used to solve the problem with coefficient decay rates of $\sigma\in\{2, 4\}$.}
  \label{fig:adapt-process-p3}
\end{figure}

We conclude the numerical observations with Figures~\ref{fig:adapt-process-p1} and~\ref{fig:adapt-process-p3} to display the behaviour of the fully adaptive algorithm.
To allow for more insights, we redefine the physical mesh resolution as m-dofs($W_N$), which is the number of FE degrees of freedom for the respective parametric solution $W_N$.
Moreover, we define the number of degrees of freedom incorporated in the operator~\eqref{eq:operator} generated to obtain the current solution.
For a tensor train operator of dimension $\mathbf{A}\in\mathbb{R}^{(N\times N)\times (q_1 \times q_1) \times \ldots\times (q_L\times q_L)}$ and rank $s=(s_1, \ldots, s_L)$, we define
\begin{equation}
\text{op-dofs}(\mathbf{A}) := N^2 s_1 - s_1^2 + \sum_{\ell=1}^{L-1} (s_\ell q_\ell^2 s_{\ell+1} - s^2_{\ell+1}) + s_L	q_L^2.
\end{equation}
Note that, using a sparse representation of the operator in the first (physical) component, it is possible to reduce number of op-dofs.

Tables~\ref{fig:adapt-process-p1} and~\ref{fig:adapt-process-p3} highlight some iteration steps and the employed stochastic expansion length $M$, the maximal polynomial degree $d^{\textrm{max}}$ (which was usually naturally attained in the first stochastic component), the maximal tensor train rank $r^{\textrm{max}}$ (again most of the time this rank is the separation rank of the physical space and the first stochastic dimension) and the corresponding degrees of freedom.
Notably, for the fast coefficient decay $\sigma=4$, the stochastic expansion is very short since most of the randomness is due to the first few stochastic parameters.
It is reasonable that the respective parameter dimensions require higher polynomial degrees in the approximation.
For the more involved slow decay $\sigma=2$ setting we observe a larger stochastic expansion of the solution and larger tensor ranks.


\footnotesize{
\subsubsection*{Acknowledgments}
M.P. and R.S. were supported by the DFG project ERA Chemistry; additionally R.S. was supported through Matheon by the Einstein Foundation Berlin.
}

\ifwias
	\newpage
\fi
\bibliographystyle{plain}
\bibliography{references}

\end{document}